\documentclass[11pt]{article}
\usepackage{amsfonts,amssymb,amsthm,amsmath,cite,bbm}
\usepackage{mathptmx}
\usepackage{enumerate}
\usepackage{enumitem}
\usepackage{hyperref}
\usepackage{comment} 
\usepackage{xcolor}
\usepackage{mathtools}
\usepackage{enumitem}
\usepackage{geometry}
\hypersetup{
	colorlinks,%
	citecolor=black,%
	filecolor=black,%
	linkcolor=black,%
	urlcolor=black
}

\newtheorem{theorem}{Theorem}[section]
\newtheorem{lemma}[theorem]{Lemma}
\newtheorem{proposition}[theorem]{Proposition}
\newtheorem{corollary}[theorem]{Corollary}
\newtheorem{remark}[theorem]{Remark}
\newtheorem{remark_example}[theorem]{Remark and example}
\newtheorem{definition}[theorem]{Definition}
\numberwithin{equation}{section}

\newcommand{\N}{{\mathbb N}} 
\newcommand{\R}{{\mathbb R}}
\newcommand{\Rn}{{\mathbb R}^n}
\newcommand{\Rnn}{{\mathbb R}^{n+1}}
\newcommand{\s}{\mathbb{S}}

\newcommand{\K}{{\mathcal K}}
\newcommand{\Kn}{{\mathcal K}^n}

\newcommand{\KN}{{\mathcal K}^{n+1}}

\newcommand{\rec}{\mathrm{rec}}

\newcommand{\fconv}{{\mbox{\rm Conv}(\R^n)}}
\newcommand{\Conv}{\mathrm{Conv}}

\newcommand{\fconvs}{{\mbox{\rm Conv}_{{\rm sc}}(\R^n)}} 
\newcommand{\fconvf}{{\mbox{\rm Conv}(\R^n; \R)}}
\newcommand{\fconvcd}{{\mbox{\rm Conv}_{{\rm cd}}(\R^n)}} 
\newcommand{\fconvl}{{\mbox{\rm Conv}^*_{\rm cd}(\R^n)}}

\newcommand{\fconvLip}{{\mbox{\rm Conv}_{{\rm Lip}}(\R^n;\R)}} 

\newcommand{\sq}{\mathbin{\vcenter{\hbox{\rule{.3ex}{.3ex}}}}} 

\newcommand{\PVal}{\mathrm{PVal}}
\newcommand{\Val}{\mathrm{Val}}

\newcommand{\PV}{\mathrm{PV}}

\newcommand{\GW}{\mathrm{GW}}

\newcommand{\interior}{\operatorname{int}} 
\newcommand{\dom}{\operatorname{dom}} 
\newcommand{\epi}{\operatorname{epi}} 
\newcommand{\supp}{\operatorname{supp}} 

\newcommand{\cone}{\Gamma}
\newcommand{\A}{\mathcal{A}}
\newcommand{\Oc}{\mathcal{O}}
\DeclareMathOperator{\vol}{vol}

\newcommand{\Sym}{\operatorname{Sym}} 

\title{Polynomial valuations on convex functions and their maximal extensions}
\author{Jonas Knoerr and Jacopo Ulivelli}
\date{}

\begin{document}

\maketitle

\begin{abstract}
    Extension problems for polynomial valuations on different cones of convex functions are investigated. It is shown that for the classes of functions under consideration, the extension problem reduces to a simple geometric obstruction on the support of these valuations. The results rely on a homogeneous decomposition for the space of polynomial valuations of bounded degree and the support properties of certain distributions associated to the homogeneous components. As an application, an explicit integral representation for valuations of top degree is established.
    
   
    \bigskip

{\noindent
{\bf 2020 AMS subject classification:} 52B45, 26B25, 53C65\\
{\bf Keywords:} convex function, convex body, valuation, continuous extension, polynomial, distribution
}
\end{abstract}

\section{Introduction}

Let $\mathcal{F}$ be a family of (extended) real-valued functions. A functional $Z:\mathcal{F}\rightarrow V$ with values in a vector space $V$ is called a valuation if 
\begin{align}\label{eq:valuation_on_functions}
	Z(f\vee h)+Z(f\wedge h)=Z(f)+Z(h)
\end{align}
for all $f,h\in \mathcal{F}$ such that the pointwise maximum $f\vee h$ and pointwise minimum $f\wedge h$ also belong to $\mathcal{F}$. Over the last ten years, there has been a substantial effort to characterize a variety of well-known functionals on different spaces of functions as the unique valuations with given invariance properties. We refer to \cite{ColesantiLombardiValuationsspacequasi2017,ColesantiEtAlTranslationinvariantvaluations2018,MussnigValuationslogconcave2021, KoneValuationsOrliczspaces2014,LiMaLaplacetransformsvaluations2017,LudwigValuationsfunctionspaces2011,LudwigValuationsSobolevSpaces2012,ColesantiEtAlclassinvariantvaluations2020,ColesantiEtAlContinuousvaluationsspace2021, LudwigFisherinformationmatrix2011,MaRealvaluedvaluations2016,WangSemivaluations$BVRn$2014,BaryshnikovEtAlHadwigersTheoremdefinable2013,TradaceteVillanuevaValuationsBanachlattices2020,TsangValuationsLpspaces2010,TsangMinkowski2012} for some of these classification results.\\
In this article, we consider valuations on subsets of the space 
\begin{align*}
	\fconv \coloneqq\{f: \R^n \rightarrow \R\cup\{\infty\}\colon f \text{ is convex, lower semi-continuous, and proper} \},
\end{align*}
where we call an extended real-valued function proper if it is not identical to $\infty$. This space carries a natural metrizable topology induced by epi-convergence (compare Section \ref{sec:preliminaries}), and we will consider its subspaces with the induced topology. Due to their intimate relation to convex bodies, valuations on convex functions have been the focus of intense research \cite{ColesantiEtAlValuationsconvexfunctions2019, ColesantiEtAlMinkowskivaluationsconvex2017, ColesantiEtAlTranslationinvariantvaluations2018, ColesantiEtAlHessianValuations2020, ColesantiEtAlHadwigertheoremconvex2020, ColesantiEtAlhomogeneousdecompositiontheorem2020, ColesantiEtAlHadwigertheoremconvex2021, ColesantiEtAlHadwigertheoremconvex2022, ColesantiEtAlHadwigertheoremconvex2022a, HugMussnigUlivelli1, HugMussnigUlivelli2, MussnigVolumepolarvolume2019, Mussnig$SLn$invariantvaluations2021, Knoerrsupportduallyepi2021, KnoerrUnitarily2021, KnoerrSmoothvaluationsconvex, KnoerrSingularvaluationsHadwiger2022,KnoerrOperators2023}. Most of the literaure considers valuations defined on suitable subsets, insetad of the whole space $\fconv$. While these restricted settings are often motivated by geometric considerations (see, for example, \cite{HofstaetterSchusterBlaschkeSantaloinequalities2023}), there usually is no single natural choice. This is potentially problematic since these classification results can depend in a highly non-trivial way on the chosen subset. Consequently, the classes of valuations can differ drastically, compare \cite{MussnigVolumepolarvolume2019,Mussnig$SLn$invariantvaluations2021,HofstaetterKnoerrEquivariantendomorphismsconvex2023}. A possible approach to this issue is finding the smallest subset of functions, to obtain the largest possible class of valuations. Doing so, we shift the problem to the question whether the valuations obtained in such a classification can be extended to larger spaces of convex functions, which is in general highly non-trivial.\\

In this article, we study this extension problem for the class of \emph{polynomial valuations}. Let $\mathrm{Aff}(n)$ denote the space of all affine functions $\ell:\R^n\rightarrow\R$. We can identify this space with $\R^{n+1}$ by identifying $(y,c)\in \R^n\times\R$ with the map $\ell(x)=\langle y,x\rangle +c$. Given a set $\cone\subset \fconv$ that is closed with respect to the addition of affine functions, we call a valuation $Z:\cone\rightarrow V$ polynomial of degree at most $d\in \mathbb{N}$ if the map
\begin{align*}
	\mathrm{Aff}(n)&\rightarrow V\\
	\ell&\mapsto Z(f+\ell)
\end{align*}
is a polynomial on $\mathrm{Aff}(n)\cong \R^{n+1}$ of degree at most $d$ for every $f\in\cone$. In other words, we require that these maps are polynomials for every choice of $f\in \cone$ and that the degree of these maps is uniformly bounded independent of this choice. We will also call such a valuation a polynomial valuation for brevity without explicitly stating the bound on the degree. The main goal of this article is to show that under suitable assumptions on the set $\cone$ and the vector space $V$, the problem of extending a given continuous polynomial valuation can be solved if and only if a single geometric restriction on the support of these functionals (which we define below) is satisfied. Before we state the precise setup for our main results, let us present the outcome of this approach in the following situation: Let $\fconvf$ denote the space of all finite convex functions on $\R^n$ and consider the subset $\fconvLip\subset \fconvf$ of all globally Lipschitz continuous convex functions, which is dense in $\fconvf$. In this setting, the topology of epi-convergence is equivalent to uniform convergence on compact subsets of $\Rn$.
\begin{theorem}
	\label{maintheorem:Lipschitz}
	Let $Z:\fconvLip\rightarrow\R$ be a continuous polynomial valuation of degree at most $d$. Then $Z$ extends uniquely to a continuous polynomial valuation of degree at most $d$ on $\fconvf$.
\end{theorem}
Note that there is no restriction to the extension problem in this specific case. In stark contrast, it was shown in \cite{ColesantiEtAlhomogeneousdecompositiontheorem2020} that there exists no non-trivial (that is, non-constant) continuous and real-valued polynomial valuation of degree $0$ on $\fconv$. Thereofore, the extension problem of extending such valuations from the dense subset $\fconvf\subset\fconv$ admits only trivial solutions.

\subsection{Main results}

Let us now introduce the general framework of this article. We will be interested in cones $\cone\subset \fconv$ that contain enough functions to allow for suitable local variations. Consider the space
\begin{align*}
	\fconvcd\coloneqq \{f\in\fconv: \dom (f)~\text{is compact}\},
\end{align*}
where $\dom f=\{x\in \R^n: f(x)<\infty\}$ denotes the domain of $f\in\fconv$. This space admits a rather simple interpretation in terms of convex bodies, compare \cite{KnoerrUlivelli2023}, which can be used to interpret valuations on convex functions as valuations on convex bodies. We will be interested in the image of this space under the Fenchel-Legendre transform, which assigns to $f\in\fconv$ the convex function $f^*\in\fconv$ given by
\begin{equation}\label{eq:F-L}
	f^*(x) \coloneqq \sup_{y \in \R^n} \{\langle x, y \rangle - f(y) \}.
\end{equation}For example, the Fenchel-Legendre transform of the convex indicator of a convex set $A\subset \R^n$,
\begin{align*}
	I_A(x)=\begin{cases}
		0 & x\in A,\\
		\infty & x\notin A,
	\end{cases}
\end{align*}
is the support function $h_A(y)=\sup_{x\in A}\langle y,x\rangle$ of $A$. Under this map, $\fconvcd$ is mapped to
\begin{align*}
	\fconvl\coloneqq \{f\in\fconv: \exists K\in\KN~\text{s.t.}~f=h_K(\cdot,-1)\}.
\end{align*}
Moreover, it is easy to see that $\fconvLip$ is mapped to the subspace of $\fconv$ consisting of functions with bounded domain. In particular, $\fconvl\subset \fconvLip$. \\
We will consider valuations on the following sets of functions.
\begin{definition}
	Let $\Oc\subset \A\subset \R^n$ be two closed convex sets with non-empty interior. We call a subset $\cone\subset \Conv(\R^n)$ an $(\A,\Oc)$-cone if
	\begin{itemize}
		\item The set $\cone$ is a cone with respect to pointwise addition and multiplication with $t>0$,
		\item For every $f\in\cone$, $\Oc\subset\dom f\subset \A$,
		\item The functions $I_\A$ and $I_\Oc$ belong to $\cone$,
		\item For every $K\in\mathcal{K}^{n+1}$, the function $h_K(\cdot,-1) +I_\A$ belongs to $\cone$.
	\end{itemize}
\end{definition}\noindent
	Notice that all of the sets $\fconvf$, $\fconvLip$, and $\fconvl$ are $(\A,\Oc)$-cones for $\A=\Oc=\R^n$. In contrast, $\fconv$ is not an $(\A,\Oc)$-cone for any choice of subsets $\Oc\subset \A\subset \R^n$. In fact, in many applications, the relevant spaces are themselves not $(\A,\Oc)$-cones, but contain a variety of  $(\A,\Oc)$-cones for different choices of $\Oc\subset \A\subset \R^n$, usually as a dense subset. Consequently, by restricting a valuation on a given set to one of these $(\A,\Oc)$-cones, the results presented below can still be applied in these more general settings. In Section \ref{sec:dualr_results}, we will present a geometric interpretation of these families of functions, and corresponding version of the following results obtained by duality through the Fenchel-Legendre transform.\\

Throughout this article, $V$ will denote a Hausdorff topological vector space over the real numbers. Let $\cone\subset\fconv$ denote an $(\A,\Oc)$-cone and denote by $\PV_d(\cone,V)$ the space of all continuous valuations $Z:\cone\rightarrow V$ that are polynomial of degree at most $d \in \N $. \\

Note that any element $Z\in\PV_d(\cone,V)$ defines a valuation on $\fconvl$ by considering the map $f\mapsto Z(f+I_\A)$. It is then easy to see (compare Theorem \ref{theorem:extension}) that this defines an element of $\PV_d(\fconvl,V)$. We may consider this as an inclusion
\begin{align*}
	\PV_d(\cone,V)\hookrightarrow \PV_d(\fconvl,V).
\end{align*}
The main goal of this article is to describe the image of this map and its dependence on $\cone$ in order to describe which valuations in $\PV_d(\fconvl,V)$ can be extended to $\cone$. For $d=0$, a version of this problem was already considered by the first named author in \cite{Knoerrsupportduallyepi2021}. 
The main idea is that homogeneous valuations in $\PV_d(\cone,V)$ may be interpreted as distributions on suitable products of copies of $\R^n$. The convex sets $\A$ and $\Oc$ then impose restrictions on the support of these distributions and, under suitable assumptions on the vector space $V$, the supports of these distributions are always compact. In particular, the value of $Z(f)$ should only depend on the values of $f$ on a neighbourhood of the support of the distributions, which allows us to freely change the convex function outside of this set to obtain suitable extensions to different families of convex functions.\\

The compactness of the support of these distributions is crucial for this kind of extension procedure. The argument in \cite{Knoerrsupportduallyepi2021} relied on the existence of convex functions that diverge arbitrarily fast as one approaches the boundary of their domain. In particular, this argument does not apply to valuations on $\fconvl$, which are essentially our main objects of interest. The key observation we are going to exploit is that the compactness of these supports is not a consequence of properties of the spaces $\cone$, but rather a consequence of the choice of topology - epi-convergence.\\

In order to make this precise, we require a decomposition of $\PV_d(\cone,V)$ into homogeneous components. For $d=0$, such a decomposition was already obtained by Colesanti, Ludwig, and Mussnig \cite{ColesantiEtAlhomogeneousdecompositiontheorem2020} (see also \cite{Knoerrsupportduallyepi2021}) as a consequence of the homogeneous decomposition for translation invariant valuations by McMullen \cite{McMullenContinuoustranslationinvariant1980}. We call $Z\in \PV_d(\cone, V)$ homogeneous of degree $k$ if 
\begin{align*}
	Z(t f)=t ^kZ(f)\quad\text{for all}~t > 0,~f\in\cone.
\end{align*}
Following an analogue result on polynomial valuations on convex bodies due to  Khovanski\u{\i} and Pukhlikov \cite{PukhlikovKhovanskiiFinitely1992}, we will deduce the corresponding version for polynomial valuations on $\cone$.
\begin{theorem}
	\label{maintheorem:HomogeneousDecomposition}
	Let $Z\in \PV_d(\cone,V)$. There exist unique continuous $k$-homogeneous valuations $Z_k\in \PV_d(\cone,V)$, $k \in \{0,1,\dots, n+d\}$, such that
	\begin{align*}
		Z=\sum_{i=0}^{n+d}Z_k.
	\end{align*}
\end{theorem}\noindent
Let $\PV_{d,k}(\cone, V)$ denote the subspace of $k$-homogeneous valuations. Theorem \ref{maintheorem:HomogeneousDecomposition} thus shows that we have a direct sum decomposition
\begin{align*}
	\PV_{d}(\cone,V)=\bigoplus_{k=0}^{d+n}\PV_{d,k}(\cone,V).
\end{align*}

In \cite{GoodeyWeilDistributionsvaluations1984}, Goodey and Weil introduced a natural correspondence between continuous and translation invariant valuations on the family of convex bodies and suitable distributions on products of the unit Euclidean sphere $\s^{n-1}$. This idea was adapted by the first author in \cite{Knoerrsupportduallyepi2021, KnoerrUnitarily2021, KnoerrSmoothvaluationsconvex} to interpret valuations on $\fconvf$ as distributions on products of $\R^n$. In this work, we use the same approach for polynomial valuations on $\Gamma$. Let $\cone^{\infty}\subset \Conv(\R^n,\R)\cap C^\infty(\R^n)$ denotes the subspace of all functions $f\in \Conv(\R^n,\R)\cap C^\infty(\R^n)$ such that $f+I_\A\in \cone$. 
\begin{theorem}
	\label{maintheorem:GW}
	Let $V$ be a complete locally convex vector space and assume that $V$ admits a continuous norm.	For every $Z\in \PV_{d,k}(\cone,V)$ there exists a unique symmetric distribution $\GW(Z):C^\infty((\R^n)^k)\rightarrow V$ with compact support such that
	\begin{align}
		\label{eq:characterizingPropertyGW}
		\GW(Z)[f\otimes\dots\otimes f]=Z(f+I_\A)\quad \text{for}~f\in \cone^{\infty}.
	\end{align}
	Moreover, the support of this distribution is contained in the intersection of the interior of $\Oc^k\subset (\R^n)^k$ with the diagonal $\{(x,\dots,x)\in (\R^n)^k: x\in\R^n\}$.
\end{theorem}\noindent
We call $\GW(Z)$ the \emph{Goodey-Weil distribution} of $Z\in\PV_{d,k}(\cone,V)$.

\begin{remark_example}
	The distribution $\GW(Z)$ can still be defined even if $V$ does not admit a continuous norm. However, the support of this distribution is not necessarily compact, as illustrated by the measure-valued valuation discussed in the following example. Given a closed convex set $\Oc\subset\R^n$ with non-empty interior let $\mathcal{M}(\mathrm{int}\Oc)\coloneqq C_c(\mathrm{int}\Oc)'$ denote the space of signed Radon measures on $\mathrm{int}\Oc$ equipped with the weak*-topology. We consider the $(\Oc,\R^n)$-cone
\begin{align*}
	\Conv(\Oc;\R)\coloneqq \{f\in\fconv: \dom f=\Oc\}
\end{align*}
and the Dirichlet energy 
\begin{align*}
	D:\Conv(\Oc;\R)&\rightarrow\mathcal{M}(\mathrm{int}\Oc)\\
	f&\mapsto \left[B\mapsto \int_B|\nabla f(x)|^2dx\right].
\end{align*}
Then, $D$ defines a continuous valuation on $\Conv(\Oc;\R)$ that is polynomial of degree at most $2$. Further examples can be constructed by replacing the Lebesgue measure with certain Monge-Amp\`ere-type operators. For the Hessian measure, this gives rise to variants of the so-called $k$-Hessian energy. See, for example, \cite{CaseWangDirichlet2018}.\\
In this example, as well as in similar cases, the corresponding left-hand side of \eqref{eq:characterizingPropertyGW} is in general not well-defined. In fact, the characterizing property on the support has to be replaced by a slightly weaker notion. Nevertheless, the distribution uniquely determines the original valuation, as we will show in Corollary \ref{cor:non_cont_norm}.
\end{remark_example}

Let $Z=\sum_{k=0}^{n+d}Z_k$ be the decomposition of $Z\in \PV_d(\cone,V)$ into its homogeneous components. We define the \emph{support of $Z$} by
\begin{align*}
	\supp Z\coloneqq \bigcup_{k=1}^{n+d}\Delta_k^{-1}(\supp\GW(Z_k)),
\end{align*}
where $\Delta_k:\R^n\rightarrow (\R^n)^k$, $x\mapsto (x,\dots,x)$, denotes the diagonal embedding. Note that this is a closed subset of $\R^n$ and in fact compact if $V$ admits a continuous norm. Moreover, the support of a $0$-homogeneous valuation is empty by definition. In Proposition \ref{proposition:CharacterizationSupport}, we will provide an alternative characterization of the support. In particular, we will show that if $f, g \in \cone$ coincide on a neighbourhood of the support of $Z\in\PV_d(\cone,V)$, then $Z(f)=Z(g)$. If the support is compact, this allows us to obtain the desired extension property.
\begin{theorem}
	\label{maintheorem:extension}
	Let $V$ be a locally convex vector space that admits a continuous norm. Then for every $(\A,\Oc)$-cone $\cone$ the restriction map
	\begin{align*}
		\PV_d(\cone,V)&\rightarrow\left\{Z\in\PV_{d}(\fconvl,V):\supp Z \subset \mathrm{int}\Oc\right\}\\
		Z&\mapsto \left[f\mapsto Z(f+I_\A)\right]
	\end{align*}
	is a well-defined bijection.
    
More generally, if\, $V$ is an arbitrary locally convex vector space and $Z\in \PV_d(\fconvl,V)$ is a valuation such that $\supp Z$ is compact and contained in $\mathrm{int}\Oc$, then there exists a unique valuation $Z_{\cone}\in \PV_d(\cone,V)$ such that
	\begin{align*}
		Z(f)=Z_\cone(f+I_\A)\quad
	\end{align*}
for every $f\in\fconvl$.
\end{theorem}\noindent
In \cite[Corollary 18]{KnoerrUlivelli2023}, the authors proved this result in the special case $\cone=\fconvf$, $V=\R$, $d=0$ under the additional assumptions that the valuations are homogeneous of degree $n$. Theorem \ref{maintheorem:extension} encodes the generality of this behaviour, entailing the following consequence. 
\begin{corollary}
    \label{maincorollary:extensionRealValuations}
	Let $Z:\fconvl \rightarrow\R$ be a continuous polynomial valuation of degree at most $d$. Then $Z$ extends uniquely by continuity to a continuous valuation on $\Conv(\R^n,\R)$.
\end{corollary}\noindent
Note that Theorem \ref{maintheorem:Lipschitz} is a direct consequence of this result.
Let us remark that Theorem \ref{maintheorem:extension} can be used to directly extend various characterization results for dually epi-translation invariant valuations on $\fconvf$ (that is, polynomial valuations of degree at most $0$) to more general cones. This applies, in particular, to the classification of rotation invariant valuations in $\PV_0(\fconvf,\R)$ by Colesanti, Ludwig, and Mussnig \cite{ColesantiEtAlHadwigertheoremconvex2020}. The necessary support restrictions can be found in \cite{KnoerrSingularvaluationsHadwiger2022}. We leave the details to the interested reader.\\

For a given valuation with compact support, our construction give rise to the following extension.
\begin{corollary}
    \label{maincorollary:maximalExtension}
    Let $V$ be a locally convex vector space and $Z \in \PV_d(\fconvl, V)$ a valuation with compact support. Then $Z$ extends uniquely by continuity to a continuous valuation on 
    \begin{align*}
        \cone_Z:=\{f \in \fconv\colon \supp Z \subset \interior (\dom f)\}.
    \end{align*}
\end{corollary}\noindent
Note that this extension is maximal in the sense that if $V$ admits a continuous norm and $Z$ extends to a continuous valuation on an $(\R^n,\Oc^n)$-cone $\cone$, then necessarily $\cone\subset\cone_Z$.

As an application, we provide the following characterization for the elements of $\PV_{d,n+d}(\cone, \R)$, that is the top-degree component of the homogeneous decomposition in Theorem \ref{maintheorem:HomogeneousDecomposition}. The case $d=0$, $\cone=\fconvf$ was previously established in \cite{ColesantiEtAlhomogeneousdecompositiontheorem2020}. We identify the family of real-valued polynomials in $n+1$ variables of degree at most $d$ with the space of symmetric tensors of rank $d$ on $\Rn \times \R$, which we denote by $\Sym^d(\R^{n}\times\R)$.
\begin{theorem}
	\label{maintheorem:CharacterizationTopComponent}
	For every $Z\in\PV_{d,n+d}(\cone,\R)$ there exists a unique function $\phi\in C_c(\R^n,\Sym^d(\R^{n}\times\R))$ with $\supp\phi \subset\mathrm{int}\Oc$ such that
	\begin{align*}
		Z(f)=\int_{\Rn \times \Rn} \phi(x)[y,f(x)]d\Theta_0(f;(x,y)).
	\end{align*}
	Conversely, the right-hand side of this equation defines an element of $\PV_{d,n+d}(\cone,\R)$ for every such function $\phi$.
\end{theorem}
Here, $\Theta_0$ denotes the Hessian measure of order $0$, which is a non-negative Borel measure on $\R^n\times\R^n$, compare \cite{ColesantiEtAlHomogeneous2024}. If $f$ is smooth on the interior of its domain, then $Z(f)=\int_{\dom f} \phi(x)[\nabla f(x),f(x)]\det(D^2f(x))dx$, where $D^2f$ denotes the Hessian matrix of $f$.
\subsection{Plan of the article}
In Section \ref{sec:preliminaries}, we recall some basic notions on convex functions and collect some technical results. We also review the necessary results about polynomial valuations on convex bodies. In Section \ref{sec:homogeneous}, we establish Theorem \ref{maintheorem:HomogeneousDecomposition}. The Goodey-Weil distributions and their properties are treated in Section \ref{sec:GW}, including the proof of  Theorem \ref{maintheorem:GW}. Finally, in Section \ref{sec:extensions}, we obtain the extension Theorem \ref{maintheorem:extension} and Corollary \ref{maincorollary:maximalExtension}. Section \ref{sec:top_degree} is devoted to the representation result in Theorem \ref{maintheorem:CharacterizationTopComponent}, while we provide a geometric interpretation of our main results induced by the Fenchel-Legendre transform in Section \ref{sec:dualr_results}.

\section{Preliminaries}\label{sec:preliminaries}

When working in $\Rn$, we consider the standard scalar product $\langle \cdot,\cdot \rangle$ and the corresponding Euclidean norm $|\cdot|$. We do not specify the dimension, as it will be clear from the context. For $r > 0$, $x \in \Rn,$ we denote by $B_r(x)$ the closed Euclidean ball of radius $r$ centred at $x$. The interior of a set $A \subset \Rn$ is denoted by $\interior A$. For basic facts on convex geometry, we refer to Schneider's monograph \cite{SchneiderConvexBodiesBrunn2013}. For a comprehensive background on convex functions, we refer to the monograph by Rockafellar and Wets \cite{{RockafellarWetsVariationalAnalysis1998}}. We now collect some results that will be relevant to the constructions in this article.

\subsection{Convex functions}
    
    The spaces of convex functions considered below are all subsets of the space $\fconv$, and we will equip this space with the topology induced by epi-convergence. Here, a sequence $(f_j)_j$ in $\fconv$ epi-converges to $f\in\fconv$ if the following conditions hold for every $x\in\R^n$:
        \begin{enumerate}[itemsep=0pt]
            \item For every sequence $(x_j)_j$ in $V$ converging to $x$, $f(x)\le \liminf\limits_{j\rightarrow\infty}f_j(x_j)$,
            \item There exists a sequence $(x_j)_j$ converging to $x$ such that $f(x)= \lim\limits_{j\rightarrow\infty}f_j(x_j)$.
        \end{enumerate}
    Let us remark that the topology induced by epi-convergence is metrizable (compare, e.g., \cite[Theorem 7.58]{RockafellarWetsVariationalAnalysis1998}). We will require the following characterization of epi-convergence.
	\begin{lemma}[\cite{AttouchWetsEpigraphical1989} Theorem 7.17]
		\label{lemma:EpiConvergencePointwiseConvergence}
		Let $f\in\fconv$ be a function such that $\dom f$ has non-empty interior, $(f_j)_j$ a sequence in $\fconv$. The following are equivalent:
		\begin{enumerate}[itemsep=0pt]
			\item The sequence $(f_j)_j$ epi-converges to $f$,
			\item The sequence $(f_j)_j$ converges pointwise to $f$ on a dense subset,
			\item The sequence $(f_j)_j$ converges uniformly to $f$ on all compact subsets that do not contain a boundary point of $\dom u$.
		\end{enumerate} 
	\end{lemma}\noindent
    Now let $\cone\subset\fconv$ be an $(\A,\Oc)$-cone. Then $\Oc\subset\A$ are by assumption two non-empty closed and convex subsets of $\R^n$ that have non-empty interior. In particular, the domain of any element in $\cone$ contains the open set $\Oc$ and has thus non-empty interior.
	\begin{lemma}\label{lemma:cont_add}
		Pointwise addition defines a continuous map
		\begin{align*}
			\cone\times \cone & \rightarrow \cone\\
			(f,h) &\mapsto f+h.
		\end{align*}
		Similarly, the map
		\begin{align*}
			\cone &\rightarrow \cone\\
			f&\mapsto t f
		\end{align*}
		is continuous for every $t>0$.
	\end{lemma}
    \begin{proof}
        Let $(f_j)_j$ and $(h_j)_j$ be sequences in $\cone$ that epi-converge to $f,h\in\Gamma$ respectively. Then, $(f_j)_j$ converges pointwise to $f$ on $\R^n\setminus \partial \dom f$ and $(h_j)_j$ converges pointwise to $h$ on $\R^n\setminus \partial \dom h$ by Lemma \ref{lemma:EpiConvergencePointwiseConvergence}, which are both open and dense in $\R^n$. In particular, $(f_j+h_j)_j$ converges pontwise on the open and dense subset $(\R^n\setminus \partial \dom f)\cap(\R^n\setminus \partial \dom h)$ to $f+h$, and thus this sequence epi-converges to $f+h$ by Lemma \ref{lemma:EpiConvergencePointwiseConvergence}. This shows the first claim. The second follows with a similar argument.
    \end{proof}
	
	For an $(\A,\Oc)$-cone $\cone$ consider the map
	\begin{align*}
		H_\A \colon \KN &\rightarrow \cone\\
		K &\mapsto h_K(\cdot,-1)+I_\A.
	\end{align*} 
    Since a sequence $(K_j)_j$ converges to $K$ in the Hausdorff metric if and only if $(h_{K_j})_j$ converges to $h_K$ uniformly on compact sets, Lemma \ref{lemma:cont_add} implies the following.
	\begin{lemma}\label{lemma:HAContinuous}
		$H_\A \colon \KN \rightarrow \cone$ is continuous.
	\end{lemma}
	
	The following is a special case of \cite[Example 9.14]{RockafellarWetsVariationalAnalysis1998}.
	\begin{proposition}
		\label{proposition:LocalLipschitzConstants}
		Let $U\subset V$ be a convex open subset and $f:U\rightarrow\R$ a convex function. If $A\subset U$ is a set with $A+B_\epsilon (0)\subset U$ and $f$ is bounded on $A+ B_\epsilon(0)$, then $f$ is Lipschitz continuous on $A$ with Lipschitz constant $\frac{2}{\epsilon}\sup_{x\in A+B_\epsilon(0)}|f(x)|$.
	\end{proposition}\noindent
	Note that any $f\in\Conv(\R^n)$ is continuous on the interior of its domain by \cite[Theorem 2.35]{RockafellarWetsVariationalAnalysis1998}. In particular, this applies to any function belonging to an $(\A,\Oc)$-cone $\cone$.
 
    The following result is well-known and follows, for example, from \cite[Lemma 9]{ColesantiEtAlValuationsconvexfunctions2019} in combination with \cite[Corollary 2.10]{Knoerrsupportduallyepi2021}.
	\begin{proposition}
        \label{proposition:ConvcdDense}
		Let $f\in \fconv$. There exists a sequence $(K_j)_j$ in $\KN$ such that $h_{K_j}(\cdot,-1)$ epi-converges to $f$ for $j\rightarrow\infty$.
	\end{proposition}
    This implies the following density result.
		\begin{lemma}\label{lemma:dual_density}
		The map $H_\A\colon \KN \rightarrow \cone$ has dense image.
	\end{lemma}
    \begin{proof}
        Let $f\in\cone$. By Proposition \ref{proposition:ConvcdDense} there exists a sequence $(K_j)_j$ in $\KN$ such that $h_{K_j}(\cdot,-1)$ converges to $f$. Since $f=f+I_\A$, the sequence $H_\A(K_j)=h_K(\cdot,-1)+I_\A$ converges to $f$ by Lemma \ref{lemma:cont_add}.
    \end{proof}

    We recall that for $f \in \fconv$ its subgradient at $x \in \R^n$ is the set \[ \partial f(x) \coloneqq \{ y \in \R^n \colon f(z) \geq f(x)+\langle z-x,y \rangle \text{ for every }z \in \R^n\}.\]
    The following properties are well-known.
    \begin{lemma}[\cite{RockafellarConvex1997} Theorem 23.5]
		\label{lemma:PropertiesSubgradients}
		For $f\in\fconv$, $x,y\in \R^n$, the following are equivalent:
		\begin{enumerate}[itemsep=0pt]
			\item $y\in\partial f(x)$,
			\item $\langle y,x\rangle =f(x)+f^*(y)$,
			\item $y\in \mathrm{argmax}_{x\in V}\langle y,x\rangle -f(x)$.
		\end{enumerate}
	\end{lemma}\noindent
	Moreover, we will make use of a special case of \cite[ 2.1.2 i]{ClarkeOptimization1990}).
	\begin{lemma}\label{lemma:SubdifferentialLipschitz}
	Let $U\subset\R^n$ be open and convex, $f:U\rightarrow\R$ convex and Lipschitz continuous on a neighbourhood of $x_0\in U$ with Lipschitz constant $L>0$. Then $\partial f(x_0)$ is non-empty, compact, and convex with $\partial f(x_0)\subset B_L(0)$.
	\end{lemma}
	Given an extended real-valued function $f: \Rn \to \R \cup \{\infty \}$ we denote its epi-graph by
\begin{align*}
	\epi(f)=\{(x,t)\in \R^n\times\R: f(x)\le t\}.
\end{align*}
    Note that $\epi(f)$ is a closed convex set in $\R^{n+1}$ for every $f\in\fconv$. 
	The following construction will be crucial for our extension procedure. It is an adaptation of \cite[Proposition 2.9]{Knoerrsupportduallyepi2021}. 
	\begin{lemma}
		\label{lemma:ReplaceFunctionByBody}
		Let $f\in \fconv$ be a function with $\interior(\dom f)\ne\emptyset$, $c>0$, $\epsilon>0$, and $A\in\K^n$ with $A+ B_\epsilon(0)\subset \interior(\dom f)$ be given. If $c\ge\max_{x\in A+ B_\epsilon(0)} f(x)$, then the convex body
		\begin{align*}
			K_{f,A,\epsilon}\coloneqq\epi(f^*)\cap \left\{(y,t)\in \R^n\times\R: |y|\le\frac{2c}{\epsilon}, |t|\le c\left(1+\frac{2\max_{x\in A}|x|}{\epsilon}\right)\right\}
		\end{align*}
		is non-empty and satisfies $f\ge h_{K_{f,A,\epsilon}}(\cdot,-1)$ as well as
		\begin{align*}
			f(x)=h_{K_{f,A,\epsilon}}(x,-1)\quad\text{for all}~x\in A.
		\end{align*}
	\end{lemma}	
	\begin{proof}
		First note that
		\begin{align*}
			h_{K_{f,A,\epsilon}}(x,-1)=\sup_{(y,t)\in K_{f,A,\epsilon}}\left\{\langle y,x\rangle -t \right\}\le \sup_{(y,t)\in \epi f^*}\left\{\langle y,x\rangle -t\right\}=f(x)\quad \text{for all}~x\in \R^n,
		\end{align*}
		since $K_{f,A,\epsilon}\subset \epi f^*$. It is thus sufficient to show that the reverse inequality holds for $x\in A$. Since $f$ is locally Lipschitz continuous on $\mathrm{int}(\dom f)$ by Proposition \ref{proposition:LocalLipschitzConstants} and $A\subset \mathrm{int}(\dom f)$, $\partial f(x)$ is non-empty for every $x\in A$ by Lemma \ref{lemma:SubdifferentialLipschitz}. As $A+ B_\epsilon(0)$ is contained in $\mathrm{int}(\dom f)$ by assumption, $f$ is Lipschitz continuous on $A$ with Lipschitz constant bounded by $\frac{2}{\epsilon}\sup_{x\in A+B_\epsilon(0)}|f(x)|\le \frac{2c}{\epsilon}$. Lemma \ref{lemma:SubdifferentialLipschitz} thus shows that $|y|\le \frac{2c}{\epsilon}$ for any $y\in \partial f(x_0)$, $x_0\in A$. Moreover, Lemma \ref{lemma:PropertiesSubgradients} implies, for $y\in\partial f(x_0)$,
		\begin{align*}
			|f^*(y)|=|\langle y,x_0\rangle-f(x_0)|\le |y| \sup_{x\in A}|x| +c\le c\left(1+\frac{2\sup_{x\in A}|x|}{\epsilon}\right).
		\end{align*}
		Therefore, $(y,f^*(y))\in K_{f,A,\epsilon}$ for all $y\in \partial f(x_0)$. For $x_0\in A$ fix $y_0\in \partial f(x_0)$. Lemma \ref{lemma:PropertiesSubgradients} thus shows that
		\begin{align*}
			f(x_0)=\langle y_0,x_0\rangle -f^*(y_0)\le \sup_{(y,t)\in K_{f,A,\epsilon}}\langle y,x_0\rangle -t=h_{K_{f,A,\epsilon}}(x_0,-1).
		\end{align*} 
	\end{proof}

    For $K \in \KN$ we consider the function $\lfloor K\rfloor\in \fconvcd$ given by
    \begin{equation}\label{eq:floor_function}
         \lfloor K \rfloor(x) \coloneqq \inf_{(x,t)\in K} t,
    \end{equation}
    compare \cite[Section 3.1]{KnoerrUlivelli2023}. Then, \[ h_K(x, -1)= \lfloor K \rfloor^*(x), \quad x \in \Rn.\] 
    \begin{lemma}
        \label{lemma:localVariationsConvCD}
        Let $K\in\KN$ and $\phi: \R^n \rightarrow \R$ be bounded. If $h_K(\cdot,-1)+\phi\in\fconvf$, then this function belongs to $\fconvl$.
    \end{lemma}
    \begin{proof}
        Since $\phi$ is bounded, there exist $C, c \in \R $ such that \[ h_K(\cdot,-1)+c \leq h_K(\cdot,-1)+\phi \leq h_K(\cdot, -1) +C. \] Therefore, by the properties of the Fenchel-Legendre transform, \[ \lfloor K \rfloor - C \leq (h_K(\cdot,-1)+\phi)^* \leq \lfloor K \rfloor - c.\] In particular, $(h_K(\cdot,-1)+\phi)^*$ is a convex function with compact domain $\dom \lfloor K \rfloor $. 

        For $T\geq \max_{x \in \dom \lfloor K \rfloor} (h_K(\cdot,-1)+\phi)^*(x)$, consider the convex body \[ K_\phi=\epi((h_K(\cdot,-1)+\phi)^*) \cap \left\{(x,t)\in \R^{n+1}: x \in \R^n, t \leq T \right\}.\] Then \[ \lfloor K_\phi \rfloor= (h_K(\cdot,-1)+\phi)^*\]  by construction and since $h_{K}(\cdot,-1)+\phi\in\fconvf$ we infer \[ h_{K_\phi}(\cdot,-1)= \lfloor K_\phi \rfloor^*=((h_K(\cdot,-1)+\phi)^*)^*=h_{K}(\cdot,-1)+\phi,\] concluding the proof.
    \end{proof}
    
\subsection{Polynomial valuations on convex bodies}

    Throughout this section, $V$ denotes a Hausdorff topological vector space over the real numbers. A functional $\mu: \Kn \rightarrow V$  with values in $V$ is called a valuation if
    \begin{equation}\label{eq:valuation_on_bodies}
        \mu(K\cup L)+\mu(K \cap L)=\mu(K)+\mu(L)
    \end{equation}
    whenever $K,L,K\cup L\in \Kn$. We refer to \cite[Section 6]{SchneiderConvexBodiesBrunn2013} for a comprehensive overview of the topic. 
    
    Let $\PVal_d(\R^n,V)$ denote the space of all continuous valuations $\mu:\Kn\to V$ that are polynomial of degree at most $d$ with respect to translations, that is, such that
    \begin{align*}
        x\mapsto \mu(K+x), \qquad x \in \Rn,
    \end{align*}
    is a polynomial on $\R^n$ of degree at most $d$ for all $K\in\Kn$, where we write \[K+x = \{y + x: y \in K \}.\] The following result follows from a special case of \cite[Theorem 1]{PukhlikovKhovanskiiFinitely1992}.
    \begin{theorem}
        \label{theorem:PolynomialityPolyValBodies}
        Let $\mu\in \PVal_d(\R^n,V)$. Then $t\mapsto \mu(tK)$ is a polynomial of degree at most $n+d$ in $t\ge 0$.
    \end{theorem}
    Set $\PVal_{d,k}(\R^n,V)\coloneqq \{\mu\in\PVal_d(\R^n,V):\mu(tK)=t^k\mu(K)\,\text{for}~t > 0,K\in\Kn\}$. The following result is well-known. It follows from a standard argument, which we include for the convenience of the reader.
	\begin{theorem}
		\label{theorem:HomDecompPolyValBodies}
		$\PVal_d(\R^n,V)=\bigoplus_{k=0}^{n+d}\PVal_{d,k}(\R^n,V)$.
	\end{theorem}
    \begin{proof}
        Let $\mu\in \PVal_d(\R^n,V)$. Theorem \ref{theorem:PolynomialityPolyValBodies} implies that there exist $\mu_k:\Kn\rightarrow V$ for $k \in \{0,1,\dots,n+d\}$ such that
        \begin{align*}
            \mu(tK)=\sum_{k=0}^{d+n}t^k \mu_k(K)\quad\text{for all}~K\in\Kn,~t\ge 0.
        \end{align*}
        Note that $\mu_k(K)$ is uniquely determined by this equation.Fix now some positive values $0<t_0<\dots< t_{n+d}$. Plugging these values in the equation above and using the inverse of the Vandermonde matrix, we obtain unique constants $c_{kj}\in \R$ independent of $K$ such that
        \begin{align*}
            \mu_k(K)=\sum_{k=0}^{n+d}c_{kj}\mu(t_j K)\quad\text{for}~K\in \Kn.
        \end{align*}
        This implies in particular that $\mu_k\in \PVal_d(\R^n,V)$. Moreover, comparing coefficients in $t$ in the expression
        \begin{align*}
            \sum_{k=0}^{d+n}t^k \mu_k(\lambda K)=\mu(t\lambda K)=\sum_{k=0}^{d+n}t^k\lambda^k \mu_k(K),
        \end{align*}
        we see that $\mu_k(\lambda K)=\lambda^k\mu_k(K)$. Thus, $\mu_k$ is $k$-homogeneous, which completes the proof.
    \end{proof}
    
    We will need the following characterization of $n$-homogeneous translation invariant valuations due to Hadwiger \cite{HadwigerVorlesungen1957}. More precisely, Hadwiger established the result for real-valued valuations. However, the proof can be easily adapted to the more general setting. The version of this result given below is also a simple consequence of results by McMullen \cite{McMullenWeaklyContinuousValuations1983}. We include a sketch of proof for the convenience of the reader. We denote by $\vol_n$ the standard $n$-dimensional Lebesgue measure on $\Rn$.
    \begin{theorem}
        \label{theorem:VolumeCharacterization}
        Let $\mu\in \PVal_{0,n}(\R^n,V)$. Then there exists $v\in V$ such that $\mu(K)=\vol_n(K) v$ for all $K\in\Kn$.
    \end{theorem}
    \begin{proof}
        The restriction of $\mu$ to polytopes defines a weakly continuous and translation invariant valuation in the sense of \cite{McMullenWeaklyContinuousValuations1983}. Since $\mu$ is homogeneous of degree $n$, its restriction to lower dimensional convex bodies vanishes due to Theorem \ref{theorem:HomDecompPolyValBodies}. Thus, by \cite[Theorem 1]{McMullenWeaklyContinuousValuations1983} there exists $v\in V$ such that $\mu(P)=\vol_n(P)v$ for all polytopes $P\in\Kn$. The valuation $K\mapsto \mu(K)-\vol_n(K)v$ vanishes on polytopes. Therefore, it vanishes on a dense subspace and thus vanishes identically since it is continuous and $V$ is Hausdorff. Thus, $\mu(K)=\vol_n(K)v$
    \end{proof}
    The next result, which should be well known, is a direct consequence of Theorems \ref{theorem:HomDecompPolyValBodies} and \ref{theorem:VolumeCharacterization}. We report a proof for the convenience of the reader. We identify the space of homogeneous polynomials of degree $d$ with coefficients in $V$ with $V \otimes \Sym^d(\R^n)$. More precisely, if $P: \Rn \rightarrow V$ is such polynomial, we consider the identification \[ \sum_{|\alpha|=d} \eta_{\alpha} x^\alpha \longmapsto \sum_{|\alpha|=d} \eta_\alpha \otimes e_\alpha,\] where $\alpha=(\alpha_1,\dots,\alpha_n), \alpha_i \in \{0,\dots,d\},$ is a multi-index and the vectors $e_\alpha $ form the standard base of $\Sym^d(\Rn)$. Under this identification, $e_\alpha$ corresponds to the monomial $x\mapsto e_\alpha(x)\coloneqq x^\alpha$.
	\begin{theorem}
		\label{theorem:TopComponentPolyValBodies}
		For every $\mu\in\PVal_{d,n+d}(\R^n,V)$ there exists a unique polynomial $P=\sum_{|\alpha|=d}\eta_\alpha\otimes e_\alpha\in V \otimes \Sym^d(\R^n)$ such that 
		\begin{align*}
			\mu(K)=\int_KP(x)\, dx\coloneqq \sum_{|\alpha|= d}\left(\int_{K}e_\alpha(x)dx\right)\cdot \eta_\alpha
		\end{align*}
        for all $K\in\Kn$.
	\end{theorem}
    \begin{proof}
        Since $\mu\in \PVal_d(\R^n,V)$, there exist maps (the coefficients of the homogeneous components of the polynomial associated to $\mu$) $\nu_i:\Kn\rightarrow V\otimes \Sym^i(\R^n)$ such that
        \begin{align*}
            \mu(K+x)=\sum_{i=0}^d\nu_i(K)[x].
        \end{align*}
        Fix $0<t_0<\dots<t_d$. Then
        \begin{align*}
            \mu(K+t_jx)=\sum_{i=0}^dt_j^i\nu_i(K)[x].
        \end{align*}
        Therefore, using the inverse of the Vandermonde matrix again, we obtain constants $c_{ij}$ independent of $K\in\Kn$ and $x\in\R^n$ such that
        \begin{align*}
            \nu_i(K)[x]=\sum_{j=0}^d c_{ij}\mu(K+t_jx).
        \end{align*}
        In particular, $\nu_i:\Kn\rightarrow V \otimes \Sym^i(\R^n)$ is a continuous valuation. Notice that
        \begin{align*}
            \sum_{i=0}^d\nu_i(K)[x+x']=\mu(K+x+x')=\sum_{i=0}^d\nu_i(K+x)[x'].
        \end{align*}
        Considered as a polynomial in $x'\in\R^n$, the highest order term is given by $\nu_d(K)[x']$ for the left-hand side of this equation, while it is given by $\nu_d(K+x)[x']$ for the right-hand side. Consequently, $\nu_d$ is translation invariant. If $t>0$, then
        \begin{align*}
            t^{n+d}\sum_{i=0}^d\nu_i(K)[x]=t^{n+d}\mu(K+x)=\mu(tK+tx)=\sum_{i=0}^dt^i\nu_i(tK)[x]
        \end{align*}
        since $\mu$ is homogeneous of degree $n+d$. Comparing coefficients as polynomials in $x$, we obtain $t^{n+d}\nu_d(K)[x]=t^d\nu_d(t K)[x]$, which shows that $\nu_d(tK)=t^n\mu(K)$. In total, we obtain $\nu_d\in \Val_n(\R^n, V \otimes \Sym^d(\R^n))$. Due to Hadwiger's volume characterization in Theorem \ref{theorem:VolumeCharacterization}, we obtain a polynomial $P=\sum_{|\alpha|=d}\eta_\alpha\otimes e_\alpha\in V \otimes \Sym^d(\R^n)$ such that
        \begin{align*}
            \nu_d(K)[x]=P(x)\vol_n(K).
        \end{align*}
        
        Consider the valuation $\tilde{\mu}\in \PVal_{d,n+d}(\R^n,V)$ defined by
        \begin{align*}
            \tilde{\mu}(K)=\int_{K}P(x)\, dx=\sum_{|\alpha|= d}\left(\int_{K}e_\alpha(x)dx\right)\cdot \eta_\alpha.
        \end{align*}
        Then $\tilde{\mu}$ defines a continuous polynomial valuation of degree at most $d$. Moreover, it is easy to see that $\tilde{\mu}(K+x)$ and $\nu_d(K)[x]$ agree up to lower order terms in $x$. In particular, $\mu-\tilde{\mu}\in \PVal_{d-1,n+d}(\R^n,V)$ is a polynomial valuation of degree at most $d-1$. Since $\PVal_{d-1,n+d}(\R^n,V)$ is trivial by Theorem \ref{theorem:HomDecompPolyValBodies}, this implies $\mu=\tilde{\mu}$.
    \end{proof}

\subsection{From valuations on convex functions to valuations on convex bodies}

    Let $V$ denote a Hausdorff topological vector space over the real numbers.
    In the remainder of the article, we will always assume that $\cone$ is an $(\A,\Oc)$-cone for $\Oc \subseteq \A \subset \Rn$ fixed closed convex sets with non-empty interior.
	\begin{lemma}
		\label{lemma:injective}
		For $Z:\cone\rightarrow V$ define $\mu:\KN\rightarrow V$ by $\mu=Z\circ H_\A$. Then the following holds:
		\begin{enumerate}[itemsep=0pt]
            \item If $Z$ is a valuation, then so is $\mu$,
			\item If $Z$ is continuous, then so is $\mu$,
            \item If $Z$ is polynomial of degree at most $d$, then so is $\mu$,
            \item If $Z$ is homogeneous of degree $k$, then so is $\mu$.
        \end{enumerate}
        In particular, the following holds if $Z$ is continuous:
        \begin{enumerate}[itemsep=0pt]
        \item[(a)]\label{lemma:injective1} If $\mu=0$, then $Z=0$.
		\item[(b)] The functional $Z$ is homogeneous of degree $k$ if and only if $\mu$ is homogeneous of degree $k$.
		\end{enumerate}
	\end{lemma}
	\begin{proof}
		\begin{enumerate}

            \item Let $K_1,K_2 \in \KN$ such that $K_1\cup K_2 \in \KN$. By the definition of $H_\A$,
  \begin{align*}
      H_\A(K_1)\vee H_\A(K_2)&=(h_{K_1}(\cdot,-1)+I_\A)\vee (h_{K_2}(\cdot,-1)+I_\A) \\
      &=\left( h_{K_1}(\cdot,-1)\vee h_{K_2}(\cdot,-1)\right)+I_\A=h_{K_1 \cup K_2}(\cdot,-1)+I_\A\\
      &=H_\A(K_1 \cup K_2).
  \end{align*}
  analoguely,
  \begin{align*}
      H_\A(K_1)\wedge H_\A(K_2)&=(h_{K_1}(\cdot,-1)+I_\A)\wedge (h_{K_2}(\cdot,-1)+I_\A)  \\
      &=\left( h_{K_1}(\cdot,-1)\wedge h_{K_2}(\cdot,-1)\right)+I_\A=h_{K_1 \cap K_2}(\cdot,-1)+I_\A\\
      &=H_\A(K_1 \cap K_2),
  \end{align*}
  proving \textit{1}. 
            \item This is a direct consequence of Lemma \ref{lemma:HAContinuous}.
			\item Let $X=(v,t)\in \R^n\times \R$. Then,
				\begin{align*}
					H_\A(K+X)[x]=h_{K+X}(x,-1)+I_\A=h_{K}(x,-1)+\langle v,x\rangle -t+I_\A,
				\end{align*}
				and therefore
				\begin{align*}
					Z\circ H_\A (K+X)=Z(H_\A(K)+\langle v,\cdot\rangle -t).
				\end{align*} 
				In particular, if $Z$ is polynomial of degree at most $d$, then the same applies to $Z\circ H_\A$. 
			\item Assume that $Z$ is $k$-homogeneous. Since $H_\A(tK)=h_{tK}(\cdot,-1)+I_\A=th_{K}(\cdot,-1)+tI_\A=tH_\A(K)$ for $t>0$, 
			\begin{align*}
				Z\circ H_\A(tK)=Z(tH_\A(K))=t^kZ(K)
			\end{align*}
			for all $K\in\KN$. Thus $Z\circ H_\A$ is $k$-homogeneous.
		\end{enumerate}

    Assume now that $Z$ is continuous.
    \begin{itemize}
        \item[(a)]
			Since the image of $H_\A$ is dense in $\cone$ by Lemma \ref{lemma:dual_density}, $Z\circ H_\A=0$ implies that $Z$ vanishes on a dense set and thus has to vanish identically by continuity.

        \item[(b)] Assume that $Z\circ H_\A$ is $k$-homogeneous. For $t>0$ consider the valuation $Z_t:\Gamma\rightarrow V$ defined by $Z_t(f)=Z(tf)-t^kZ(f)$, which is continuous by Lemma~\ref{lemma:cont_add}. Then $Z_t\circ H_\A=0$ since $Z\circ H_\A$ is $k$-homogeneous. Thus, $Z_t=0$ for all $t>0$ by (a), which shows that $Z$ is $k$-homogeneous.
    \end{itemize}
		
	\end{proof}
		
\section{Homogeneous decomposition}\label{sec:homogeneous}
Throughout this section, $V$ denotes a Hausdorff topological vector space over the real numbers.
We now prove a homogeneous decomposition result, following the strategy in \cite[Theorem 1]{Knoerrsupportduallyepi2021}. Instead of McMullen's homogeneous decomposition Theorem \cite{McMullenContinuoustranslationinvariant1980}, we use the polynomial variant given by Theorem \ref{theorem:HomDecompPolyValBodies}.
\begin{proposition}\label{prop:hom_dec}
    For every $Z \in \PV_d(\Gamma,V)$, there exist unique $Z_k \in \PV_{d,k}(\cone,V)$, $k \in \{0,1,\dots,n+1+d\}$ such that \[ Z=Z_0+\cdots+Z_{n+1+d}.\] 
    In particular, 
    \begin{align*}
        \PV_d(\cone, V)=\bigoplus_{k=0}^{n+1+d}\PV_{d,k}(\cone, V).
    \end{align*}
\end{proposition}
\begin{proof}
    For $t> 0$, consider the continuous polynomial valuations $Z^t$ defined by $Z^t(u)=Z(t u)$.
    By Lemma \ref{lemma:injective}, we have that 
    \begin{align*}
        \mu^t \colon \KN &\rightarrow V\\
        K &\mapsto Z^t(H_\A(K))
    \end{align*}
    is a continuous and polynomial valuation on $\KN$ of degree at most $d$. Thus, by Theorem \ref{theorem:HomDecompPolyValBodies} there exist unique  valuations $\mu_k$, $k \in \{0,\dots,n+1+d\}$, which are polynomial of degree at most $d$, $k$-homogeneous, continuous, and such that
    \begin{equation}\label{eq:hom_dec}
        \mu^t=\sum_{k=0}^{n+1+d} t^k \mu_k.
    \end{equation}
    Choose $0< t_0<\cdots <t_{n+1+d}$. Plugging these values into \eqref{eq:hom_dec} and considering the inverse of the Vandermonde matrix, we can find values $c_{ki}\in\R$ for $k=0,\dots,n+1+d$ such that \[\mu_k=\sum_{i=0}^{n+1+d} c_{ki}\mu^{t_i}.\]

    If we define $Z_k\coloneqq \sum_{i=0}^{n+1+d} c_{ki}Z^{t_i}$, then $\mu_k=Z_k\circ H_\A$, which shows that $Z_k$ is uniquely determined by $\mu_k$. Note that the valuations \[Z_k^t-t^k Z_k,\quad t>0, \,~k=0,\dots,n+1+d,\] vanish on the image of $H_\A$ by construction. Thus, by Lemma \ref{lemma:injective1}, $Z_k^t=t^k Z_k$, so $Z_k$ is homogeneous of degree $k$. Similarly, the valuation $Z-\sum_{k=0}^{n+d+1}Z_k$ vanishes on the image of $H_\A$, and thus has to vanish identically. Thus $Z=\sum_{k=0}^{n+d+1}Z_k$. It is clear that the decomposition is unique.
\end{proof}

\begin{proposition}\label{proposition:PolynomialTopComponentWRTDualTranslations}
	For $Z\in \PV_{d,k}(\cone,V)$, let $Y_i:\cone\rightarrow V\otimes \Sym^{i}(\mathrm{Aff}(n))$ be the unique functions with
    \begin{align*}
        Z(f+\ell)=\sum_{i=0}^d Y_i(f)[\ell]\quad\text{for}~f\in\cone,~\ell\in \mathrm{Aff}(n).
    \end{align*}
    Then  $Y_i\in \PV_{d-i,k-i}(\cone,V \otimes \Sym^i(\mathrm{Aff}(n)))$ and
	\begin{align*}
		\ell \mapsto \left( Z(f+\ell)-Y_d(f) \right) [\ell], \qquad \ell \in \mathrm{Aff}(n),
	\end{align*}
	is a polynomial of degree at most $d-1$. In particular, if $Y_d=0$, then $Z$ is polynomial of degree at most $d-1$.
	\end{proposition}
	\begin{proof}
		By definition,
		\begin{align*}
			Z(f+t\ell)=\sum_{i=0}^{d}t^iY_i(f)[\ell]\quad \text{for}~t\in\R.
		\end{align*}
		
		Plugging in $t=1,\dots, d+1$ and using the inverse of the Vandermonde matrix, we obtain constants $c_{ij}$ independent of $f$ and $\ell$ such that
		\begin{align*}
			Y_i(f)[\ell]=\sum_{i=1}^{d+1}c_{ij}Z(f+j\ell).
		\end{align*}
		The right-hand side is a continuous polynomial valuation, so the same holds for $Y_i$. Thus $Y_i\in \PV_{d}(\cone,V\otimes \Sym^{d-i}(\mathrm{Aff}(n))$. Moreover, for $t>0$,
		\begin{align*}
			\sum_{i=0}^{d}t^{k}Y_i(f)[\ell]=t^kZ(f+\ell)=Z(tf+t\ell)=\sum_{i=0}^{d}Y_i(tf)[t\ell]=\sum_{i=0}^{d}t^{i}Y_i(tf)[\ell].
		\end{align*}
		Comparing coefficients, we see that $Y_i$ is homogeneous of degree $k-i$. Next, note that for $\ell,\ell':\R^n\rightarrow \R$ affine,
		\begin{align*}
			\sum_{i=0}^{d}Y_i(f+\ell')[\ell]=Z(f+\ell'+\ell)=\sum_{i=0}^{d}Y_i(f)[\ell+\ell'].
		\end{align*}
		As a polynomial in $\ell$, the $i=d$ homogeneous component is given by $ Y_d(f+\ell')[\ell]$ on the left-hand side, while it is given by $ Y_d(f)[\ell]$ on the right-hand side of this equation. Comparing coefficients, we thus obtain $ Y_d(f+\ell')[\ell]= Y_d(f)[\ell]$, and therefore $ Y_d$ is invariant with respect to dual epi-translations, that is, $ Y_d\in \PV_{0,k-d}(\cone, V \otimes \Sym^d(\mathrm{Aff}(n)))$. Moreover, 
		\begin{align*}
			Z(f+\ell)-Y(f)[\ell]=\sum_{i=0}^{d-1}Y_i(f)[\ell]
		\end{align*}
		is a polynomial of degree at most $d-1$ in $\ell$. Repeating this argument, we obtain that $Y_i\in \PV_{d-i,k-i}(\cone,V \otimes \Sym^i(\mathrm{Aff}(n)))$ by induction.
	\end{proof}
    Theorem \ref{maintheorem:HomogeneousDecomposition} now follows from Proposition \ref{prop:hom_dec} and the following result.
    \begin{corollary}\label{cor:TrivialComponent}
		If $Z\in \PV_{d,n+1+d}(\cone,V)$, then $Z=0$.
	\end{corollary}
    \begin{proof}
        By Proposition \ref{proposition:PolynomialTopComponentWRTDualTranslations} there exists a unique $Y\in \PV_{0,n+1}(\cone,V\otimes\Sym^d(\mathrm{Aff}(n)))$ such that
        \begin{align*}
            Z(f+\ell)-Y(f)[\ell]
        \end{align*}
        is a polynomial of degree at most $d-1$ in $\ell\in\mathrm{Aff}(n)$. We will show that $Y=0$, so $Z$ is polynomial of degree at most $d-1$, i.e. $Z\in \PV_{d-1,n+d+1}(\cone,V))$, which is trivial by Proposition \ref{prop:hom_dec}. Set $\tilde{V}\coloneqq V\otimes \Sym^d(\mathrm{Aff}(n))$ and consider $\mu\coloneqq Y\circ H_\A$. Then $\mu\in \PVal_{0,n+1}(\R^{n+1},\tilde{V})$ by Lemma \ref{lemma:injective}. Moreover, it is sufficient to show that $\mu=0$. By Hadwiger's volume characterization in Theorem \ref{theorem:VolumeCharacterization}, there exists $v\in \tilde{V}$ such that $\mu(K)=v\vol_{n+1}(K)$ for all $K\in\KN$. Let $K_R=B_1(0)\times[0,R]\subset \R^n\times\R$ be given for $R\ge0$. Since $H_\A(K_R)=H_\A(K_0)$ by construction, we obtain 
        \begin{align*}
            0=v\vol_{n+1}(K_0)=\mu(K_0)=\mu(K_R)=vR\vol_{n}(B_1(0))
        \end{align*}
        for $R>0$, so $v=0$. In particular, $\mu=0$.
    \end{proof}
Let us add some consequences of this result that will be used in later sections.
\begin{corollary}
    \label{corollary:VanishingOnIndicators}
    Let $Z\in \PV_{d,k}(\cone,V)$, $k\ge 1$. If $I_A\in \cone$ for some closed convex set $A\subset \R^n$, then $Z(I_A)=0$.
\end{corollary}
\begin{proof}
    This follows because $t I_A=I_A$ for all $t>0$, so $Z(I_A)=Z(tI_A)=t^k(I_A)$ for all $t>0$, which implies $Z(I_A)=0$ by taking the limit $t\rightarrow 0$.
\end{proof}

The following result follows from Proposition \ref{prop:hom_dec} using a standard argument, compare \cite[Section 4.2]{Knoerrsupportduallyepi2021}.
\begin{corollary}
    \label{corollary:existencePolarization}
	For every $Z\in\PV_{d,k}(\cone, V)$ there exists a unique continuous map $\overline{Z}:\cone^k\rightarrow V$ such that
	\begin{enumerate}[itemsep=0pt]
		\item $\overline{Z}$ is symmetric,
		\item $\overline{Z}$ is additive in each argument: for $f,h,f_2,\dots,f_k\in\cone$, $t>0$:
		\begin{align*}
			\overline{Z}(t f+h,f_2,\dots,f_k)=t\overline{Z}(f,f_2,\dots,f_k)+\overline{Z}(h,f_2,\dots,f_k)
		\end{align*}
		\item $\overline{Z}(f,\dots,f)=Z(f)$ for $f\in\cone$.
	\end{enumerate}
\end{corollary}
The map $\overline{Z}$ is called the \emph{polarization} of $Z\in \PV_{d,k}(\cone, V)$. In the next section we will extend this functional to a distribution, which will require the following estimate. Let $|\cdot|_V$ denote a semi-norm on $V$. For a set $A\subset \cone$ we define for $Z \in \PV_d(\cone,V)$
\begin{align*}
    \Vert Z \rVert_{V;A}\coloneqq \sup_{f\in A}\lvert Z(f)\rvert_V.
\end{align*}
In general, $\lVert\mu\rVert_{V;A}\in[0,\infty]$, however, if the semi-norm is continuous and $A$ is a compact subset of $\cone$, then $\lVert\mu\rVert_{V;A}$ is finite due to the continuity of $\mu$. The following is an adaptation of \cite[Lemma 4.13]{Knoerrsupportduallyepi2021}. Since the proof is analogue, we omit it.
\begin{lemma}\label{lemma:estimatePolarization}
	There exists a constant $C_k>0$ with the following property: If $|\cdot|_V$ is a semi-norm on $V$ and $A\subset \Gamma$ is convex with $I_\A\in A$, then
    \[\lVert \overline{Z}\rVert_{V;A} \coloneqq \sup_{g_1,\dots,g_k \in A} |\overline{Z}(g_1,\dots,g_k)|_V\leq C_k \lVert Z \rVert_{V;A}\] for every $Z \in \PV_{d,k}(\Gamma,V)$.
\end{lemma}

\section{Goodey-Weil distributions}\label{sec:GW}
    From this section on, $V$ will denote a locally convex vector space over the real numbers. We will first lift the polarization $\overline{Z}$ of $Z\in\PV_{d,k}(\cone,V)$ to a continuous multilinear functional on $C^\infty_c(\R^n)$. We require the following result.
	\begin{lemma}
		\label{lemma:HelpLemmaGW}
		Let $\phi\in C^\infty_c(\R^n)$ with $\supp\phi\subset B_R(0)$. Then $\|\phi\|_{C^2(\R^n)}(1+R^2)^{3/2}h_{B_1(0)}(\cdot,-1)+\phi$ is convex and belongs to $\fconvl$.
	\end{lemma}
    \begin{proof}
        Since $h_{B_1(0)}(x,-1)=\sqrt{1+|x|^2}$, it is easy to see that the Hessian of this function is positive definite. Thus $\|\phi\|_{C^2(\R^n)}(1+R^2)^{3/2}h_{B_1(0)}(\cdot,-1)+\phi\in\fconvf$, and therefore belongs to $\fconvl$ by Lemma \ref{lemma:localVariationsConvCD}.
    \end{proof}
	Given $Z\in\PV_{d,k}(\cone,V)$, we define $\tilde{Z}:(C_c^\infty(\R^n))^k\rightarrow V$ as follows: For $\phi_1,\dots,\phi_k\in C_c^\infty(\R^n)$ we choose $f_i\in \cone$ such that $f_i+\phi\in \cone$, for example by adding the indicator $I_\A$ to one of the functions given by Lemma \ref{lemma:HelpLemmaGW}. We then define
	\begin{align*}
		\tilde{Z}(\phi_1,\dots,\phi_k)\coloneqq\sum\limits_{l=0}^k(-1)^{k-l}\frac{1}{(k-l)!l!}\sum_{\sigma\in S_k}\overline{Z}(f_{\sigma(1)}+\phi_{\sigma(1)},\dots,f_{\sigma(l)}+\phi_{\sigma(l)},f_{\sigma(l+1)},\dots,f_{\sigma(k)}).
	\end{align*}
    Since $\overline{Z}$ is additive in each argument, it is easy to see that $\tilde{Z}(\phi_1,\dots,\phi_k)$ does not depend on the choice of the functions $f_i$ and is linear in each argument.
	\begin{corollary}
        \label{corollary:EstimateGW}
		Let $\cone$ be an $(\A,\Oc)$-cone. There exists a constant $C_k>0$ and a compact set $A\subset \cone$ with the following property: If $|\cdot|_V$ is a continuous semi-norm on $V$ and $\phi_1,\dots,\phi_k\in C^\infty_c(\R^n)$ are supported on $B_R(0)$, then 
		\begin{align*}
			|\tilde{Z}(\phi_1,\dots,\phi_k)|_{V}\le C_k(1+R^2)^{3k/2} \prod_{i=1}^{k}\|\phi_i\|_{C^2(\R^n)} \lVert Z \rVert_{V,A}.
		\end{align*}
  In particular, $\tilde{Z}:C^\infty_c(\R^n)^k\rightarrow V$ is continuous.
	\end{corollary}
	\begin{proof}
		Set $A\coloneqq H_\A(\{K\in\KN: K\subset B_{2}(0)\})$. Then $A$ is compact as the image of a compact set under the continuous map $H_\A$. Fix $R>0$. For $\phi_1,\dots,\phi_k\in C_c^\infty(\R^n)$ consider the functions
		\begin{align*}
			f_i\coloneqq h_{B_1(0)}(\cdot,-1)+\frac{\phi_i}{\|\phi_i\|_{C^2(\R^n)}(1+R^2)^{3/2}}+I_\A
		\end{align*}
		Using Lemma \ref{lemma:localVariationsConvCD}, it is easy to see that $f_i\in A$. Set $h_i\coloneqq H_\A(B_1(0))\in A$. Then
		\begin{align*}
			|\tilde{Z}(\phi_1,\dots,\phi_k)|_V\le& \sum\limits_{l=0}^k\frac{1}{(k-l)!l!}\sum_{\sigma\in S_k}|\overline{Z}(f_{\sigma(1)},\dots,f_{\sigma(l)},h_{\sigma(l+1)},\dots,h_{\sigma(k)})|_V\prod_{i=1}^k(1+R^2)^{3/2}\|\phi_i\|_{C^2(\R^n)}\\
			\le &2^k\sup_{g_1,\dots,g_k\in A} |\overline{Z}(g_1,\dots,g_k)|_V (1+R^2)^{3k/2}  \prod_{i=1}^k\|\phi_i\|_{C^2(\R^n)}.
		\end{align*}
		Since $\Gamma$ is an $(\A,\Oc)$-cone and $A$ is compact and convex with $I_\A\in A$, Lemma \ref{lemma:estimatePolarization} implies that there exists a constant $C_k$ such that \[ \sup_{g_1,\dots,g_k\in A} |\overline{Z}(g_1,\dots,g_k)| \leq C_k \lVert Z \rVert_{V;A},\] which is finite as $Z$ is continuous. This concludes the proof.
	\end{proof}
    Let us denote the completion of $V$ by $\overline{V}$. Let $\mathcal{D}(\R^n,\overline{V})$ denote the space of $\overline{V}$-valued distributions on $\R^n$, that is the space of all continuous linear maps $T:C^\infty_c(\R^n)\rightarrow \overline{V}$.
	\begin{theorem}
		For every $Z\in \PV_{d,k}(\cone,V)$ there exists a unique distribution $\GW(Z)\in \mathcal{D}((\R^n)^k,\overline{V})$ such that
		\begin{align}
			\GW(Z)[\phi_1\otimes\dots\otimes\phi_k]=\tilde{Z}(\phi_1,\dots,\phi_k)
		\end{align}\label{eq:char_GW}
		for all $\phi_1,\dots,\phi_k\in C^\infty_c(\R^n)$.
	\end{theorem}
	\begin{proof}
		Corollary \ref{corollary:EstimateGW} implies that $\tilde{Z}:C^\infty_c(\R^n)^k\rightarrow V$ is continuous and multilinear. The claim thus follows from the Schwartz Kernel Theorem (see, for example, \cite{Gask60}).
	\end{proof}		
    The construction immediately implies the following restriction on the support of these distributions.
    \begin{corollary}
        \label{corollary:ResSupportPreliminary}
        Let $V$ be a locally convex vector space (that does not necessarily admit a continuous norm) and $\cone$ be an $(\A,\Oc)$-cone. For every $Z\in \PV_{d,k}(\cone,V)$ the support of $\GW(Z)$ is contained in the closed convex set $\Oc^k$.
    \end{corollary}
    \begin{proof}
        It is sufficient to show that $\GW(Z)[\phi_1\otimes\dots\otimes \phi_k]=0$ if $\supp \phi_1\subset \R^n\setminus\Oc$. Choose $f_j\in \cone$ such that $f_j+\phi_j$ is convex. Then, the same applies to the function $f_j+I_\Oc\in\cone$. Moreover, $f_1+I_\Oc+\phi_1=f_1+I_\Oc=f_1+I_\Oc+0$. Hence, by the definition of $\tilde{Z}$,
        \begin{align*}
            &\GW(Z)[\phi_1\otimes\dots\otimes \phi_k]=\tilde{Z}(\phi_1,\dots,\phi_k)=\tilde{Z}(0,\phi_2,\dots,\phi_k)=0.
        \end{align*}
    \end{proof}
  
	\subsection{Properties of the support}

    In \cite[Theorem 5.5]{Knoerrsupportduallyepi2021} it is proved that the support of the Goodey-Weil distribution of $Z \in \PV_{0,k}(\fconvf)$ is contained in the diagonal in $(\R^n)^k$. The same proof applies in our more general setting, with a minor modification: The function $1+|x|^2$, which is in general not contained in $\cone$, has to be replaced by $\sqrt{1+|x|^2}+I_\A(x)$. Since the rest of the proof is identical, we leave the modification to the reader and only state the result.
    \begin{theorem}\label{theorem_support_GW_diagonal}
    Let $\cone$ be an $(\A,\Oc)$-cone. For $Z \in \PV_{d,k}(\cone,V), k \geq 2$, the support of $\GW(Z)$ is contained in the diagonal in $(\R^n)^k$.
    \end{theorem}\noindent
    It will be convenient to consider the support of this distribution as a subset of $\R^n$. Abusing notation, we will identify the support of $\GW(Z)$ for $Z\in\PV_{d,k}(\R^n,V)$ with the set $\Delta_k^{-1}(\GW(Z))\subset \R^n$, where $\Delta_k:\R^n\rightarrow(\R^n)^k$ denotes the diagonal embedding. In this interpretation, Corollary \ref{corollary:ResSupportPreliminary} shows that the support of $\GW(Z)$ is contained in $\Oc$.\\

    Next, we prove the first part of Theorem \ref{maintheorem:GW}, which extends the first part of \cite[Theorem 5.7]{Knoerrsupportduallyepi2021}. We remark that, although similar in spirit, this generalization significantly strengthens the original argument since it does not rely on the existence of convex functions in $\cone$ that have diverging lower bounds on their Hessians. 
	\begin{theorem}\label{thm:comp_supp}
		Assume that $V$ admits a continuous norm and let $\cone$ be an $(\A,\Oc)$-cone. For every $Z\in \PV_{d,k}(\cone,V)$ the support of $\GW(Z)$ is compact and contained in the interior of $\Oc$.
	\end{theorem}
	\begin{proof}
        Let $|\cdot|_V$ be a continuous norm on $V$, and denote by $U_\varepsilon(x)$ the open Euclidean ball centred at $x\in\R^n$ with radius $\varepsilon>0$. If the support of $\GW(Z)$ is not compact and contained in $\mathrm{int}\Oc$, then there exists a sequence $(x_j)_j$ in $\Oc$ and a sequence $(\epsilon_j)_j$, $0<\epsilon_j\le 1$, such that $U_{\epsilon_j}(x_j)\subset \Oc$ and either
    \begin{itemize}
        \item[(a)] $(|x_j|)_j$ is increasing and diverges,
        \item[(b)] $(x_j)_j$ converges to a point $x_0\in\partial \Oc$,
    \end{itemize}
    and such that there exists functions $\phi_j\in C^\infty_c(\R^n)$ with $\supp\phi_j\subset U_{\epsilon_j}(x_j)$ with $\GW(Z)[\phi_j\otimes\dots\otimes \phi_j]\ne 0$ for all $j\in\mathbb{N}$. Here we used that $\supp\GW(Z)$ is contained in the diagonal in $(\R^n)^k$ and that $\GW(Z)$ is symmetric. Since $|\cdot|_V$ is a norm, we may thus assume that 
    \begin{align*}
        |\GW(Z)[\phi_j\otimes\dots\otimes \phi_j]|_V=1\quad \text{for every}~j.
    \end{align*}
    Notice that in case (b) the sequence $(\epsilon_j)_j$ has to converge to $0$ for $j\rightarrow\infty$. Passing to a subsequence if necessary, we may assume that $\epsilon_j<\frac{1}{j}$.\\

    We handle case (a) first. Consider the convex function 
    \begin{align*}
    \tilde{f}_j\coloneqq \max\left( h_{B_1(0)}(x,-1)-\sqrt{1+(|x_j|-1)^2},0\right).
    \end{align*}
    Notice that $\tilde{f}_j\in \fconvl$ as it is the maximum of the restriction of two support functions. Moreover, by construction, $\tilde{f}_j$ is twice differentiable on a neighborhood of $U_{\epsilon_j}(x_j)$ and has positive definite Hessian. In particular, there exists $c_j>0$ such that $c_j\tilde{f}_j+\phi_j$ is convex, which therefore belongs to $\fconvl$ by Lemma \ref{lemma:localVariationsConvCD}.\\
    
    Now consider $f_j\coloneqq c_j\tilde{f}_j+I_{\Oc}\in \cone$. If $x\in \Oc$ belongs to the interior of $\Oc$, then $f_j$ coincides with the zero function on a neighborhood of $x$ for all $j\in\mathbb{N}$ large enough. If $x\in \R^n\setminus \Oc$, then $f_j(x)=\infty$ for all $j\in\mathbb{N}$. In particular, $f_j$ epi-converges to $I_\Oc$ for $j\rightarrow\infty$. The same reasoning applies to the function $f_j+\phi_j\in \cone$. We thus obtain
    \begin{align*}
        \lim_{j\rightarrow\infty}|\GW(Z)[\phi_j\otimes\dots\otimes\phi_j]|_V= & \lim_{j\rightarrow\infty}\left|\sum_{l=0}^k(-1)^{k-l}\overline{Z}\left((f_j+\phi_j)[l],f_j[k-l]\right)\right|_V\\ =&|\overline{Z}(I_{\Oc},\dots,I_{\Oc})|_V\cdot\left|\sum_{l=0}^k(-1)^{k-l}\right|=0,
    \end{align*}
    which contradicts $|\GW(Z)[\phi_j\otimes\dots\otimes\phi_j]|_V=1$ for all $j\in\mathbb{N}$. Here we used that $\overline{Z}(I_{\Oc},\dots,I_{\Oc})=Z(I_{\Oc})=0$, compare Corollary \ref{corollary:VanishingOnIndicators}.\\

    Let us prove case (b). In this case, the strategy consists of finding a suitable sequence of functions that converge to the indicator function of a hyperplane tangent to $\Oc$. We consider the sequence of points \[ z_j=x_j+\frac{x_j-x_0}{|x_j - x_0|}\left(j+ \frac{1}{j}\right).\] Each $z_j$ is the center of a ball of radius $j$ at distance $\frac{1}{j}$ from $x_j$ in the direction $x_j-x_0$. We consider the functions \[\tilde{h}_j(x)\coloneqq \max\left(\sqrt{1+|x-z_j|^2}-\sqrt{1+j^2},0\right).\] Since $x_j$ has distance $\frac{1}{j}>\epsilon_j$ from the ball $B_j(z_j)$, $\tilde{h}_j$ is twice differentiable on a compact neighbourhood of $U_{\epsilon_j}(x_j)$ with positive definite Hessian. Therefore, there exists $c_j>0$ such that $c_j\tilde{h}_j+\phi_j$ is convex. By Lemma \ref{lemma:localVariationsConvCD}, these functions belong to $\fconvl$. Consider the function $h_j\in \cone$ defined by
    \begin{align*}
        h_j\coloneqq c_j\tilde{h}_j+I_\Oc.
    \end{align*}
    Note that $h_j,h_j+\phi\in\cone$. By construction, $h_j$ is equal to $0$ on $ \Oc\cap B_{j}(z_j)$. Since $x_j\rightarrow x_0$ and $\epsilon_j\rightarrow 0$, for any $x\in \Oc$ we can find $N\in\mathbb{N}$ such that $x\in B_{j}(z_j)$ for all $j\ge N$. In particular, $h_j$ epi-converges to $I_\Oc$ for $j\rightarrow\infty$. The same argument applies to $h_j+\phi_j$. Therefore, we obtain again 
    \begin{align*}
        \lim_{j\rightarrow\infty}|\GW(Z)[\phi_j\otimes\dots\otimes\phi_j]|_V=&\lim_{j\rightarrow\infty}\left|\sum_{l=0}^k(-1)^{k-l}\overline{Z}\left((f_j+\phi_j)[l],f_j[k-l]\right)\right|_V\\
        =&|\overline{Z}(I_{\Oc},\dots,I_{\Oc})|_V\cdot\left|\sum_{l=0}^k(-1)^{k-l}\right|=0,
    \end{align*}
    which contradicts $|\GW(Z)[\phi_j\otimes\dots\otimes\phi_j]|_V=1$ for all $j\in\mathbb{N}$. Here we used again that $\overline{Z}(I_{\Oc},\dots,I_{\Oc})=Z(I_{\Oc})=0$.
	\end{proof}

	Recall that $\cone^{\infty}\subset \Conv(\R^n,\R)\cap C^\infty(\R^n)$ denotes the subspace of all smooth convex functions $f$ such that $f+I_\A\in \cone$. The following result completes the proof of Theorem \ref{maintheorem:GW}.
	\begin{theorem}
		\label{theorem:RelationGWandValuation}
		Let $V$ be a locally convex vector space that admits a continuous norm, $Z\in \PV_{d,k}(\cone,V)$. Let $\Oc\subset A\subset \A$ be a closed convex set with $I_A\in \cone$. For all $f_1,\dots,f_k\in\cone^{\infty}$,
		\begin{align*}
			\GW(Z)[f_1\otimes\dots\otimes f_k]=\overline{Z}(f_1+I_A,\dots,f_k+I_A).
		\end{align*}
	\end{theorem}
	\begin{proof}        
         Fix a point $x_0\in\mathrm{int}A$. For $\epsilon\in (0,1)$ set
		\begin{align}\label{eq:approx_domain}
			A_\delta\coloneqq [(1-\delta) [A-x_0]+x_0]\cap B_{\frac{1}{\delta}}(x_0).
		\end{align}
		Then $A_\delta\in\Kn$ is contained in $\mathrm{int}A$, and the family $(A_\delta)_\delta$ is ordered by inclusion with $\bigcup_{\delta>0}A_\delta=\mathrm{int}A$. Since $A_\delta$ is contained in the interior of $\mathrm{int}A$, there exists $\epsilon_\delta>0$ such that $A_\delta+\epsilon_\delta B_1(0)$ is contained in $\mathrm{int}A$ as well. 
        Fix $\phi_\delta\in C^\infty_c(\R^n)$ with the following properties:
        \begin{enumerate}[itemsep=0pt]
            \item $\supp\phi_\delta\subset A_{\delta/2}$,
            \item $0\le \phi_\delta\le 1$,
            \item $\phi_\delta\equiv 1$ on $A_{\delta}$.
        \end{enumerate}
        Since the support of $\GW(Z)$ is compact and contained in $\mathrm{int}\Oc\subset \mathrm{int}A$, we have
		\begin{align*}
			\GW(Z)[f_1\otimes\dots \otimes f_k]=\GW(Z)[(\phi_\delta f_1)\otimes\dots \otimes (\phi_\delta f_k)]
		\end{align*}
		for $\delta>0$ small enough. Note that 
		\begin{align*}
			D^2(\phi_\delta f_i)[x]=\phi_\delta(x) D^2f_i(x)+\left[\nabla \phi_\delta(x)\cdot \nabla f_i(x)^T+\nabla f_i(x)\cdot \nabla \phi_\delta(x)^T\right]+ f_i(x)D^2\phi_\delta(x).
		\end{align*}
        Using Lemma \ref{lemma:ReplaceFunctionByBody}, choose $K_{i,\delta}\in \KN$ for each $\delta>0$ such that $f_i=h_{K_{i,\delta}}(\cdot,-1)$ on $A_{\delta/2}$, and note that $\phi_\delta f_i$ is supported on $A_{\delta/2}$ by construction. It follows that
  \begin{align*}
			&h_{i,\delta}(x)\coloneqq  H_A(K_{i,\delta})+\\ &\left[\|\phi\|_{C^2(\R^n)} \left(\sup_{|z|\le \frac{2}{\delta}}|f_i(z)|+2|\nabla f_i(z)|\right)\left( 1+\frac{4}{\delta^2} \right)^{3/2}  \right]\max \left(\sqrt{1+|x|^2}-\sqrt{1+\frac{1}{\delta^2}},0\right)
		\end{align*}
		is a convex function such that $h_{i,\delta}+\phi_\delta f_i$ is convex. Notice that \[\max \left(\sqrt{1+|x|^2}-\sqrt{1+\frac{1}{\delta^2}},0\right)=h_{\mathrm{conv}\left(B_1(0)-\left(0,\sqrt{1+\frac{1}{\delta^2}}\right),0\right)}(x,-1).\] That is, $\max \left(\sqrt{1+|x|^2}-\sqrt{1+\frac{1}{\delta^2}},0\right)$ is the support function of a convex body. Therefore, $h_{i,\delta}$ and $h_{i,\delta}+\phi_\delta f_i$ belong to the image of $H_A$ by Lemma \ref{lemma:localVariationsConvCD}. In particular, they belong to $\cone$.
        Moreover, by Lemma \ref{lemma:EpiConvergencePointwiseConvergence},
		\begin{align*}
			h_{i,\delta}&\rightarrow f_i+I_A, \\
			h_{i,\delta}+\phi_\delta f_i&\rightarrow 2f_i+I_A=2(f_i+I_A).
		\end{align*}
		Indeed, they converge pointwise on the interior of $A$ and are equal to $\infty$ on $\R^n\setminus A$. From the definition of $\GW(Z)$ and the continuity of $\overline{Z}$, we thus obtain
		\begin{align*}
			&\GW(Z)[f_1\otimes\dots\otimes f_k]=\lim\limits_{\delta\rightarrow0}\GW(Z)[(\phi_\delta f_1)\otimes\dots \otimes (\phi_\delta f_k)]\\
			=&\lim\limits_{\delta\rightarrow0}\sum\limits_{l=0}^k(-1)^{k-l}\frac{1}{(k-l)!l!}\sum_{\sigma\in S_k}\overline{Z}(h_{\sigma(1),\delta}+\phi_\delta f_{\sigma(1)},\dots,h_{\sigma(l),\delta}+\phi_\delta f_{\sigma(l)},h_{\sigma(l+1),\delta},\dots,h_{\sigma(k),\delta})\\
			=&\sum\limits_{l=0}^k(-1)^{k-l}\frac{2^l}{(k-l)!l!}\sum_{\sigma\in S_k}\overline{Z}(f_{\sigma(1)}+I_A,\dots f_{\sigma(k)}+I_A)\\
			=&\overline{Z}(f_{1}+I_A,\dots f_{k}+I_A)\sum\limits_{l=0}^k\binom{k}{l}(-1)^{k-l}2^l=\overline{Z}(f_{1}+I_A,\dots f_{k}+I_A),
		\end{align*}
		since $\sum\limits_{l=0}^k\binom{k}{l}(-1)^{k-l}2^l=(2-1)^k=1$ by the binomial theorem. The claim follows.
	\end{proof}

    For the vector space $V$ we denote by $V'$ its topological dual. Composing a valuation with an element in $V'$, we obtain a real valued valuation. This allows for the following characterization of the Goodey-Weil distributions that holds for arbitrary locally convex vectorspaces.
    \begin{corollary}\label{cor:non_cont_norm}
        Let $Z\in\PV_{d,k}(\cone,V)$, $\lambda\in V'$. Then the support of $\lambda\circ \GW(Z)$ is compact and contained in $\mathrm{int}\Oc$, and for $f_1,\dots,f_k\in \cone^\infty$
        \begin{align*}
            (\lambda\circ \GW(Z))[f_1\otimes\dots\otimes f_k]=\lambda\left(\overline{Z}(f_1+I_\A,\dots,f_k+I_\A)\right).
        \end{align*}
        In particular, $\GW(Z)$ uniquely determines $Z$.
    \end{corollary}
    \begin{proof}
        As in the proof of the previous result, $\lambda\circ \GW(Z)=\GW(\lambda\circ Z)$ for $\lambda\in V'\cong \overline{V}'$. Since $\lambda\circ Z$ is a real-valued valuation, its support is compact and contained in $\mathrm{int}\Oc$ by Theorem \ref{thm:comp_supp}. Since $\overline{\lambda\circ Z}=\lambda\circ \overline{Z}$, the claim follows from Theorem \ref{theorem:RelationGWandValuation}. Finally, notice that $V'$ separates points in $V$ since $V$ is a locally convex space, so $Z$ is uniquely determined by $\GW(Z)$.
    \end{proof}

    Following the terminology in \cite{Knoerrsupportduallyepi2021}, we define the support of $Z\in \PV_{d,k}(\cone,V)$ to be the set
    \begin{align*}
        \supp Z\coloneqq \Delta_k^{-1}(\supp \GW(Z)),
    \end{align*}
    where $\Delta_k:\R^n\rightarrow(\R^n)^k$ denotes the diagonal embedding. If $Z\in\PV_d(\cone,V)$ and $Z=\sum_{k=0}^{n+d} Z_k$ is its decomposition into homogeneous components, we set
    \begin{align*}
        \supp Z\coloneqq \bigcup_{k=1}^{n+d}\supp Z_k.
    \end{align*}
    Note that the support of any element in $\PV_{d,0}(\cone,V)$ is empty by definition. The following is an adaptation of \cite[Proposition 6.3]{Knoerrsupportduallyepi2021}.
	\begin{proposition}
		\label{proposition:CharacterizationSupport}
		The support of $Z\in \PV_{d}(\cone,V)$ is minimal (with respect to inclusion) among all closed sets $A\subset \R^n$ with the following property: If $f,h\in \cone$ are two functions with $f\equiv h$ on a neighborhood of $A$, then $Z(f)=Z(h)$.
	\end{proposition}
	\begin{proof}
       		We may assume that $Z$ is $k$-homogeneous. Let us first show that $\supp Z\subset A$ for all closed sets $A$ with this property. Let $\phi_1,\dots,\phi_k\in C^\infty_c(\R^n)$ be supported on $\R^n\setminus A$. Then there exists $C>0$ such that $f_{i,t}\coloneqq C h_{B_1(0)}(\cdot,-1)+t\phi_i+I_\Oc$ is convex for all $1\le i\le k$, $|t|\le \epsilon$ for some $\epsilon>0$. By Lemma \ref{lemma:localVariationsConvCD}, these functions belong to $\cone$. Since $A$ is closed, $f_{i,t}$ coincides with $f_{i,0}$ on a neighborhood of $A$. Thus,
		\begin{align*}
			\GW(Z)[\phi_1\otimes\dots\otimes\phi_k]=&\frac{1}{k!}\frac{\partial^k}{\partial t_1\dots\partial t_k}\Big|_0Z\left(\sum_{i=1}^kf_{i,t}\right)=\frac{1}{k!}\frac{\partial^k}{\partial t_1\dots\partial t_k}\Big|_0Z\left(\sum_{i=1}^kf_{i,0}\right)=0.
		\end{align*}
		In particular, $\supp\GW(Z)\subset A^k$, which implies $\supp Z\subset A$.\\
  
		Let us show that $\supp Z$ satisfies the property stated above. Assume first that we have proved the statement for all locally convex vector spaces that admit a continuous norm. Given a locally convex vector space $V$ and $Z\in\PV_d(\cone, V)$, let $\lambda\in V'\cong \overline{V}'$ be a continuous linear functional. Then $\lambda\circ Z$ is a continuous, real-valued valuation, and it follows from the construction of the Goodey-Weil distributions that $\lambda\circ \GW(Z)=\GW(\lambda\circ Z)$. In particular, the support of $\GW(\lambda\circ Z)$ is contained in the support of $\GW(Z)$. Since $\lambda\circ Z$ is a real-valued valuation, our assumption implies $(\lambda\circ Z)(f)=(\lambda\circ Z)(h)$ for every $\lambda\in V'$. Since $V'$ separates points in $V$, this implies $Z(f)=Z(h)$.\\
  
        Let us now consider the case that $V$ admits a continuous norm, so the support of $Z$ is compactly contained in the interior of $\Oc$.\\
        Assume that $f,h\in\cone$ coincide on a neighborhood of $\supp Z$. Then $f+I_\Oc$ and $h+I_\Oc$ also coincide on a neighborhood of $\supp Z$ since the support is contained in the interior of $\Oc$. Notice that Corollary \ref{corollary:existencePolarization} implies for $t\in (0,1)$
        \begin{align*}
            Z(t f+(1-t)I_\Oc)=\sum_{i=0}^k\binom{k}{i}t^i(1-t)^{k-i}\overline{Z}(f[i],I_\Oc[k-i]),
        \end{align*}
        so since $\lim_{t\rightarrow1^-}t f+(1-t)I_\Oc=f+I_\Oc\in\cone$, the continuity of $Z$ implies
        \begin{align*}
            Z(f+I_\Oc)=\lim\limits_{t\rightarrow 1^-}Z(t f+(1-t)I_\Oc)=\overline{Z}(f[k],I_\Oc[0])=Z(f).
        \end{align*}
        Therefore, we may assume that the domain of $f$ and $h$ is $\Oc$.
        Fix a point $x_0\in\mathrm{int}\Oc$ and, for $\delta\in (0,1)$, consider the sets $(O_\delta)_\delta$ defined as in \eqref{eq:approx_domain} for the set $\Oc$. By Lemma \ref{lemma:ReplaceFunctionByBody} there thus exist convex bodies $K_{f,O_\delta,\delta_\epsilon},K_{h,O_\delta,\delta_\epsilon}\in\KN$ such that 
		\begin{align*}
			&f=h_{K_{f,O_\delta,\delta_\epsilon}}(\cdot,-1) \, \text{ and } \, 
			h=h_{K_{h,O_\delta,\delta_\epsilon}}(\cdot,-1)
		\end{align*} 
		on $O_\delta$. In particular, $f_\delta\coloneqq h_{K_{f,O_\delta,\delta_\epsilon}}(\cdot,-1)$ and $h_\delta\coloneqq h_{K_{h,O_\delta,\delta_\epsilon}}(\cdot,-1)$ satisfy 
		\begin{align*}
			f=\lim\limits_{\delta\rightarrow0^+}f_\delta+I_\Oc,\\
			h=\lim\limits_{\delta\rightarrow0^+}h_\delta+I_\Oc,
		\end{align*}
		since the functions on the right-hand side converge pointwise outside of the boundary of $\Oc$.
		
		Since $\supp Z$ is contained in the interior of $\Oc$, $h_{K_{f,O_\delta,\delta_\epsilon}}(\cdot,-1)$ and $h_{K_{h,O_\delta,\delta_\epsilon}}(\cdot,-1)$ coincide on a neighborhood of $\supp Z$ for $\delta>0$ small enough. Consider the functions $h_{K_{f,O_\delta,\delta_\epsilon}}, h_{K_{h,O_\delta,\delta_\epsilon}}$ on $\R^{n+1}$. Convolving them with an appropriate sequence of smooth, compactly supported, non-negative functions on $\mathrm{SO}(n+1)$ with support shrinking to $\{Id\}$, we can obtain a sequence of convex and $1$-homogeneous functions on $\R^{n+1}$ that converge to $h_{K_{f,O_\delta,\delta_\epsilon}}, h_{K_{h,O_\delta,\delta_\epsilon}}$ respectively and are smooth on $\R^{n+1}\setminus\{0\}$. As they are $1$-homogeneous and convex, these are support functions of convex bodies in $\R^{n+1}$, so their restriction to $\R^n\times\{-1\}$ defines elements of $\fconvl$. Moreover, since $h_{K_{f,O_\delta,\delta_\epsilon}}(\cdot,-1)$ and $h_{K_{h,O_\delta,\delta_\epsilon}}(\cdot,-1)$ coincide on a neighborhood of $\supp Z$ for sufficiently small $\delta>0$, the functions $h_{K_{f,O_\delta,\delta_\epsilon}}$ and $h_{K_{h,O_\delta,\delta_\epsilon}}$ coincide on a neighborhood of $\supp Z\times\{-1\}$ in $\R^{n+1}$ since they are $1$-homogeneous. Thus, the functions obtained from the convolution also coincide on a (smaller) neighborhood of $\supp Z\times\{-1\}$. In total, we obtain sequences $f_{\delta,j},h_{\delta,j}\in\fconvl$ of smooth functions converging to $f_\delta\coloneqq h_{K_{f,O_\delta,\delta_\epsilon}}(\cdot,-1)$ and $h_\delta\coloneqq h_{K_{h,O_\delta,\delta_\epsilon}}(\cdot,-1)$ respectively such that $f_{\delta,j}=h_{\delta,j}$ on a neighborhood of $\supp Z$. Thus Theorem \ref{theorem:RelationGWandValuation} implies
		\begin{align*}
			Z(f_\delta+I_\Oc)=&\lim\limits_{j\rightarrow\infty}Z(f_{\delta,j}+I_\Oc)=\lim\limits_{j\rightarrow\infty}\GW(Z)[f_{\delta,j}^{\otimes k}]\\=&\lim\limits_{j\rightarrow\infty}\GW(Z)[h_{\delta,j}^{\otimes k}]=\lim\limits_{j\rightarrow\infty}Z(h_{\delta,j}+I_\Oc)=Z(h_\delta+I_\Oc)
		\end{align*}
		for all $\delta>0$ small enough. Since $f_\delta+I_\Oc$ and $h_\delta+I_\Oc$ converge to $f$ and $h$ respectively, we thus obtain
		\begin{align*}
			Z(f)=\lim\limits_{\delta\rightarrow0}Z(f_\delta+I_\Oc)=\lim\limits_{\delta\rightarrow0}Z(h_\delta+I_\Oc)=Z(h).
		\end{align*} 
		Thus, $\supp Z$ has the desired property.
	\end{proof}

    We conclude this section with a technical lemma which we will need for the proof of Theorem \ref{maintheorem:CharacterizationTopComponent}.
    \begin{lemma}
        \label{lemma:supportTranslationPolynomials}
        For $Z\in \PV_{d,k}(\cone,V)$ let $Y_i\in \PV_{d-i,k-i}(\cone,V\otimes \Sym^i(\mathrm{Aff}(n)))$ be the unique valuations with 
        \begin{align*}
            Z(f+\ell)=\sum_{i=0}^dY_i(f)[\ell]\quad\text{for}~f\in\cone,~\ell\in \mathrm{Aff}(n).
        \end{align*}
        Then $\supp Y_i\subset \supp Z$.
    \end{lemma}
    \begin{proof}
        As in the proof of Theorem \ref{proposition:PolynomialTopComponentWRTDualTranslations}, we may use the inverse of the Vandermonde matrix to express $Y_i$ as a linear combination of the valuations $f\mapsto Z(f+t_i\ell) $ for $0<t_1<\dots< t_{d+1}$. If $f,h\in\cone$ are two functions with $f\equiv h$ on a neighborhood of $\supp Z$, then $f+t_j\ell\equiv h+t_j\ell$ on a neighborhood of $\supp Z$ for all $\ell\in\mathrm{Aff}(n)$ and $j \in \{1,\dots,d+1\}$, so $Z(f+t_j\ell)=Z(h+t_j\ell)$ by Proposition \ref{proposition:CharacterizationSupport}. This implies $Y_i(f)=Y_i(h)$ for all functions $f,h$ with this property, so $\supp Y_i$ is contained in $\supp Z$ by Proposition \ref{proposition:CharacterizationSupport}.
    \end{proof}
 
\section{Extension results}\label{sec:extensions}
    In this section we will prove the extension results in Theorem \ref{maintheorem:extension}. We start with the following more general version on extensions for valuations on $\fconvl$ with compact support.
	\begin{theorem}\label{theorem:extension}
		Let $\cone$ be an $(\A,\Oc)$-cone. If $Z\in \PV_{d}(\fconvl,V)$ has compact support and satisfies $\supp Z\subset\mathrm{int}\Oc$, then there exists a unique valuation $\tilde{Z}\in\PV_{d}(\cone,V)$ such that $Z(f)=\tilde{Z}(f+I_\A)$ for all $f\in \fconvl$.
	\end{theorem}
	\begin{proof}
		Since functions of the form $f+I_\A$, $f\in\fconvl$, are dense in $\cone$ by Lemma \ref{lemma:dual_density}, it is clear that such a valuation is unique if it exists.	\\
		Let us construct the extension. For $f\in\cone$, choose a function $\tilde{f}\in\fconvl$ that satisfies $f=\tilde{f}$ on a neighborhood of $\supp Z$, which exists due to Lemma \ref{lemma:ReplaceFunctionByBody} as $f$ is finite on the interior of $\Oc$. We then define $\tilde{Z}:\cone\rightarrow V$ by
		\begin{align*}
			\tilde{Z}(f)\coloneqq Z(\tilde{f}).
		\end{align*}
		Note that this is independent of the choice of $\tilde{f}$ due to Proposition \ref{proposition:CharacterizationSupport}. In particular $\tilde{Z}(f+I_\A)=Z(f)$ for all $f\in\fconvl$. It is thus sufficient to show that $\tilde{Z}\in\PV_{d}(\cone,V)$.
		
		Let us show that $\tilde{Z}$ is polynomial of degree at most $d$. For $\ell:\R^n\rightarrow\R$ affine, the function $f+\ell$ coincides with $\tilde{f}+\ell$ on a neighborhood of the support of $Z$, so
		\begin{align*}
			\tilde{Z}(f+\ell)=Z(\tilde{f}+\ell),
		\end{align*}
		and the right-hand side of this equation is a polynomial of degree at most $d$ in $\ell$. Thus $\ell\mapsto \tilde{Z}(f+\ell)$ is a polynomial of degree at most $d$ in $\ell$ as well.
		
		Next we show that $\tilde{Z}$ is a valuation. Let $f,h\in \cone$ be two functions such that $f\vee h,f\wedge h\in\cone$. Fix a compact and convex set $A\subset \mathrm{int}\Oc$ containing $\supp Z$ in its interior. Set 
		\begin{align*}
			c\coloneqq\max\left(\sup_{x\in A} |f(x)|,\sup_{x\in A} |h(x)|\right).
		\end{align*}
		Lemma \ref{lemma:ReplaceFunctionByBody} implies that the convex bodies
		\begin{align*}
			K_f\coloneqq&\epi(f^*)\cap \left\{(y,t)\in \R^n\times\R: |y|\le\frac{2c}{\epsilon}, |t|\le c\left(1+\frac{2\max_{x\in A}|x|}{\epsilon}\right)\right\},\\
			K_h\coloneqq&\epi(h^*)\cap \left\{(y,t)\in \R^n\times\R: |y|\le\frac{2c}{\epsilon}, |t|\le c\left(1+\frac{2\max_{x\in A}|x|}{\epsilon}\right)\right\},\\
			K_{f\vee h}\coloneqq&\epi((f\vee h)^*)\cap \left\{(y,t)\in \R^n\times\R: |y|\le\frac{2c}{\epsilon}, |t|\le c\left(1+\frac{2\max_{x\in A}|x|}{\epsilon}\right)\right\},\\
			K_{f\wedge h}\coloneqq&\epi((f\wedge h)^*)\cap \left\{(y,t)\in \R^n\times\R: |y|\le\frac{2c}{\epsilon}, |t|\le c\left(1+\frac{2\max_{x\in A}|x|}{\epsilon}\right)\right\}
		\end{align*}
		are non-empty and their support functions satisfy
		\begin{align*}
			&h_{K_f}(\cdot,-1)=f, & h_{K_h}(\cdot,-1)=h, && h_{K_{f\vee h}}(\cdot,-1)=f\vee h, &&h_{K_{f\wedge h}}(\cdot,-1)=f\wedge h \quad \text{on}~ A,
		\end{align*}
		which is a neighborhood of $\supp Z$. Since $f\wedge h$ is convex, we have
		\begin{align*}
			&\epi((f\vee h)^*)=\epi(f^*)\cup \epi(h^*), &\epi((f\wedge h)^*)=\epi(f^*)\cap \epi(h^*),
		\end{align*}
		which implies $K_f\cup K_h=K_{f\wedge h}$ and $K_f\cap K_h=K_{f\vee h}$. Thus
		\begin{align*}
			&h_{K_f}(\cdot,-1)\vee h_{K_h}(\cdot,-1) =h_{K_{f\vee h}}(\cdot,-1), &&h_{K_f}(\cdot,-1)\wedge h_{K_h}(\cdot,-1) =h_{K_{f\wedge h}}(\cdot,-1)
		\end{align*}
		due to the properties of support functions.	Since $Z$ is a valuation on $\fconvl$, we thus obtain
		\begin{align*}
			\tilde{Z}(f)+\tilde{Z}(h)=&Z(h_{K_f}(\cdot,-1))+Z(h_{K_f}(\cdot,-1))=Z(h_{K_{f\vee h}}(\cdot,-1))+Z(h_{K_{f\wedge h}}(\cdot,-1))\\
			=&\tilde{Z}(f\vee h)+\tilde{Z}(f\wedge h).
		\end{align*}
		 Thus $\tilde{Z}$ is a valuation.
   
		It remains to see that $\tilde{Z}$ is continuous. Assume to the contrary that it is not continuous. Then we can find a sequence $(f_j)_j$  in $\cone$ that epi-converges to $f\in\cone$ such that $|\tilde{Z}(f_j)-\tilde{Z}(f)|_V\ge\epsilon$ for some $\epsilon>0$ and some continuous semi-norm on $V$. Consider the compact neighborhood $A$ of $\supp Z$, which is contained in $\mathrm{int}\Oc$ by assumption. Lemma \ref{lemma:EpiConvergencePointwiseConvergence} implies that the sequence $(f_j)_j$ converges uniformly to $f$ on $A$, so we may in particular assume that $f_j$ is finite on $A$. In particular, there exists $c>0$ such that $\sup_{x\in A} |f_j(x)|\le c$ for all $j\in\mathbb{N}$. Lemma \ref{lemma:ReplaceFunctionByBody} shows that the convex bodies
		\begin{align*}
			K_{f_j}\coloneqq\epi(f_j^*)\cap \left\{(y,t)\in \R^n\times\R: |y|\le\frac{2c}{\epsilon}, |t|\le c\left(1+\frac{2\max_{x\in A}|x|}{\epsilon}\right)\right\}
		\end{align*}
		are non-empty and satisfy $f_j=h_{K_{f_j}}(\cdot,-1)$ on a neighborhood of $\supp Z$ for all $j\in\mathbb{N}$. Note that the sequence $(K_{f_j})_j$ is bounded in $\KN$. By the Blaschke selection theorem we may thus assume that the sequence converges to some body $K\in\KN$, passing to a subsequence if necessary. For $x\in A$, Lemma \ref{lemma:EpiConvergencePointwiseConvergence} thus implies
		\begin{align*}
			f(x)=\lim_{j\rightarrow\infty}f_j(x)=\lim_{j\rightarrow\infty}h_{K_j}(x,-1)=h_K(x,-1),
		\end{align*}
		so $h_K(\cdot,-1)$ and $f$ coincide on a neighborhood of the support of $Z$. Thus
		\begin{align*}
			\lim\limits_{j\rightarrow\infty}\tilde{Z}(f_j)=\lim\limits_{j\rightarrow\infty}Z(h_{K_j}(\cdot,-1))=Z(h_K(\cdot,-1))=\tilde{Z}(f),
		\end{align*}
		which contradicts $|\tilde{Z}(f_j)-\tilde{Z}(f)|_V\ge \epsilon$ for all $j\in\mathbb{N}$. Thus $\tilde{Z}$ has to be continuous.
		\end{proof}

        Theorem \ref{maintheorem:extension} follows from the previous result and the following corollary.
		\begin{corollary}\label{cor:ExtensionBijecitveCorrespondence}
			Let $V$ be a locally convex vector space that admits a continuous norm. The map
			\begin{align*}
				\PV_{d,k}(\cone,V)&\rightarrow\PV_{d,k}(\fconvl,V)\cap\{Z:\supp Z\subset \mathrm{int}\Oc\}\\
				Z&\mapsto \left[f\mapsto Z(f+I_\A)\right]
			\end{align*}
			is well-defined and bijective.
		\end{corollary}
		\begin{proof}
		We first prove that this map is well defined. Let $Z \in \PV_{d,k}(\cone,V)$. Given $f,g \in \fconvl$, \[ (f\vee g)+I_\A=(f+I_\A)\vee (g+I_\A) \text{ and } (f\wedge g)+I_\A=(f+I_\A)\wedge (g+I_\A). \] Moreover, pointwise addition is continuous with respect to epi-convergence and well-defined as long as the functions involved have domains whose intersection has non-empty interior. Therefore, $Z(\cdot + I_\A)$ defines a continuous valuation on $\fconvl$. Similarly, one easily verifies that this is a polynomial valuation that is homogeneous of degree $k$. Since $V$ admits a continuous norm, the support of $Z$ is a compact subset of $\interior \Oc$ by Theorem \ref{thm:comp_supp}. Now the characterization of the support in Proposition \ref{proposition:CharacterizationSupport} shows that the support of $f\mapsto Z(f+I_\A)$ is also compact and contained in $\mathrm{int}\Oc$, so the map is well defined.\\
        Note that the map is injective since the image of $H_\A$ is dense in $\cone$ by Lemma \ref{lemma:dual_density}. It is also onto by Theorem \ref{theorem:extension}, since every valuation in $\PV_d(\fconvl,V)$ has compact support contained in $\interior \Oc$. This concludes the proof.
		\end{proof}

    \begin{proof}[Proof of Corollary \ref{maincorollary:maximalExtension}]
        Since $\supp Z$ is compact, $\mathrm{conv}(\supp Z)$ is a compact and convex subset of $\R^n$. For $\delta>0$
        \begin{align*}
            \cone_{\delta}:=\{f\in\fconv: \mathrm{conv}(\supp Z)+B_\delta(0)\subset \dom f\}.
        \end{align*}
        Then 
        \begin{align*}
            \cone_Z=\bigcup_{\delta>0} \cone_{\delta}.
        \end{align*}
        Note that $\cone_{\delta}$ is an $(\R^n,\mathrm{conv}(\supp Z)+B_\delta(0))$-cone. Since $\supp Z$ is compact and contained in the interior of $\mathrm{conv}(\supp Z)+B_\delta(0)$, $Z$ extends uniquely by continuity to $Z_{\delta}\in\PV_d(\cone_{\delta},V)$ for all $\delta>0$. Note that the extensions coincide on $\cone_{\delta_1}\cap\cone_{\delta_2}$  for all $\delta_1,\delta_2>0$ due to Theorem \ref{theorem:extension}. In particular, we may define $Z_{\max}:\cone_Z\rightarrow V$ by
        \begin{align*}
            Z_{\max}(f)=Z_{\delta}(f)\quad\text{if}~f\in \cone_{\delta }.
        \end{align*}
        This obviously implies that $Z_{\max}$ is polynomial of degree at most $d$ with respect to the addition of affine functions. Note that if $f,h\in\cone_Z$ are two functions such that $f\wedge h,f\vee h\in \cone_Z$, then there exists $\delta>0$ such that $A_\delta=\mathrm{conv}(\supp Z)+B_{\delta}$ is contained in the domains of these functions. In particular, $f,h,f\wedge h$ and $f\vee h$ all belong to $\cone_\delta$. Since $Z_{\max}$ coincides with $Z_\delta$ on this set, it is a valuation. \\
        It remains to see that $Z_{\max}$ is continuous. Fix $f\in \cone_Z$ and choose $\delta>0$ such that $A_\delta=\mathrm{conv}(\supp Z)+B_\delta(0)\subset \interior\dom f$. If $(f_j)_j$ is a sequence on $\cone_Z$ converging to $f$, then this sequence converges uniformly to $f$ on $A_\delta$ by Lemma \ref{lemma:EpiConvergencePointwiseConvergence}. In particular, we may assume that $f_j$ is finite on $A_\delta$. Thus all of the functions belong to $\cone_\delta$, so the claim follows from the fact that $Z_{\max}$ is given by $Z_\delta$ on this set, which is continuous.
    \end{proof}
  
\section{Characterization of $\PV_{d,n+d}(\cone,\R)$}\label{sec:top_degree}
    In this section, we provide an extension of the following characterization result obtained by Colesanti, Ludwig, and Mussnig in \cite[Theorem 5] {ColesantiEtAlhomogeneousdecompositiontheorem2020}.
	\begin{theorem}\label{theorem:ClassificationTopDegVConv}
		For every $Z\in \PV_{0,n}(\fconvf,\R)$ there exists a unique $\phi\in C_c(\R^n)$ such that
		\begin{align*}
			Z(f)=\int_{\R^n\times\R^n}\phi(x)d\Theta_0(f;(x,y))\quad\text{for all}~f\in\Conv(\R^n,\R).
		\end{align*}
	\end{theorem} \noindent
    As mentioned in the introduction, $\Theta_0(f;\cdot)$ is the Hessian measure  of order $0$ of $f$ (see \cite{ColesantiEtAlHessianValuations2020} and the references therein). The uniqueness of the function $\phi$ in the statement of Theorem \ref{theorem:ClassificationTopDegVConv} is not stated explicitly in \cite{ColesantiEtAlhomogeneousdecompositiontheorem2020} although it is implied by the proof. For a further discussion of the uniqueness, see \cite[Theorem 1.4]{KnoerrUlivelli2023} and \cite{KnoerrUnitarily2021}. The relation between $\supp Z$ and $\supp \phi$ is made clearer by the following result.
	\begin{lemma}[\cite{KnoerrUnitarily2021} Corollary 2.11]\label{lemma:TopDegSupportDensity}
		Let $Z\in \PV_{0,n}(\fconvf,\R)$ by given by 
        \begin{equation*}
            Z(f)=\int_{\R^n\times\R^n}\phi(x)d\Theta_0(f;(x,y))
        \end{equation*} for some $\phi\in C_c(\R^n)$. Then $\supp Z \subseteq \supp\phi\subseteq \mathrm{conv}(\supp Z)$.
	\end{lemma}

    \begin{proof}[Proof of Theorem \ref{maintheorem:CharacterizationTopComponent}]
        Let $\cone$ be an $(\A,\Oc)$-cone. By Theorem \ref{maintheorem:extension}, the space $\PV_{d,n+d}(\cone,\R)$ is in natural bijection with the subspace of $\PV_{d,n+d}(\fconvl,\R)$ of all valuations with support contained in $\mathrm{int} \Oc$. In turn, these valuations are in natural bijection with the subspace of all valuations in $\PV_{d,n+d}(\fconvf)$ with the same support restriction. It is thus sufficient to show that any such valuation admits an integral representation of the desired type, where the function $\phi\in C_c(\R^n,\Sym^d(\mathrm{Aff}(n)))$ is supported in the interior of $\mathrm{int}\Oc$.\\
        Thus let $Z\in \PV_{d,n+d}(\fconvf)$ be a valuation with $\supp Z\subset \mathrm{int}\Oc$. Since the support of $Z$ is compact, its convex hull $A\coloneqq\mathrm{conv}(\supp Z)$ is compact and contained in the interior of $\Oc$.\\
        Let $Y_i\in \PV_{d-i,n+d-i}(\cone,\Sym^i(\mathrm{Aff}(n)))$ by the unique valuations with 
        \begin{align*}
            Z(f+\ell)=\sum_{i=0}^{d}Y_i(f)[\ell]\quad\text{for all}~f\in\fconvf,~\ell\in \mathrm{Aff}(n),
        \end{align*}
        compare Proposition \ref{proposition:PolynomialTopComponentWRTDualTranslations}. By Lemma \ref{lemma:supportTranslationPolynomials}, $\supp Y_i\subset \supp Z$.
        We thus obtain that the valuation $Y_d\in \PV_{0,n}(\fconvf,\Sym^d(\mathrm{Aff}(n)))$ satisfies $\supp Y_d\subset A$. Applying Theorem \ref{theorem:ClassificationTopDegVConv} to the different components of 
        \begin{align*}
            Y_d\in \PV_{0,n}(\fconvf,\Sym^d(\mathrm{Aff}(n)))\cong \PV_{0,n}(\fconvf,\R)\otimes \Sym^d(\mathrm{Aff}(n)),
        \end{align*} we deduce that there exists a function $\phi\in C_c(\R^n,\Sym^d(\mathrm{Aff}(n)))$ such that
        \begin{align*}
            Y_d(f)[\ell]=\int_{\R^n\times\R^n}\phi(x)[\ell]d\Theta_0(f,(x,y))\quad\text{for all}~f\in\fconvf,~\ell\in\mathrm{Aff}(n).
        \end{align*}
        Lemma \ref{lemma:TopDegSupportDensity} shows that $\phi$ is supported on $A$. Define $\tilde{Z}\in \PV_{d,n+d}(\fconvf)$ by 
        \begin{align*}
            \tilde{Z}(f)\coloneqq\int_{\R^n\times \Rn }\phi(x)[(y,f(x))]d\Theta_0(f; (x,y)).
        \end{align*}
        Notice that this is well-defined due to \cite[Theorem 17]{ColesantiEtAlhomogeneousdecompositiontheorem2020}.\\
        It follows that $Z-\tilde{Z}\in \PV_{d-1,n+d}(\fconvf)$. Since this space is trivial by Theorem \ref{maintheorem:HomogeneousDecomposition}, we obtain that $Z=\tilde{Z}$ has the desired representation. 
    \end{proof}

\section{Geometric formulation of the results} \label{sec:dualr_results}

As mentioned in the introduction, the results in this article admit a more geometric interpretation obtained by composing the given valuations with the Fenchel-Legendre transform. This duality has been widely employed to transfer results between different spaces of convex functions. Compare for example, \cite{ColesantiEtAlhomogeneousdecompositiontheorem2020, ColesantiEtAlHadwigertheoremconvex2020, ColesantiEtAlHadwigertheoremconvex2021, ColesantiEtAlHadwigertheoremconvex2022, ColesantiEtAlHadwigertheoremconvex2022a}. In this section, we present the outcome of this approach in our setting.

We start by recalling some geometric operations on convex functions and related key properties of the Fenchel-Legendre transform. For $f,g \in \fconv$ their infimal convolution is defined as \[ f \square g (x) \coloneqq \inf_{y \in \Rn} \{f(x)+g(x-y)\}.\] Geometrically, this corresponds to Minkowski addition of epi-graphs, that is, \[ \epi(f \square g)= \epi(f)+\epi(g).\] In particular, for $X=(x,t) \in \Rnn$ the epi-translation of $f \in \fconv$ can be written as \[\epi(f)+X=\epi(f(\cdot - x) + t).\] For $t>0$ the operation of epi-multiplication is given by $\lambda \sq f (x) = \lambda f\left( \frac{x}{\lambda}\right)$, and for $\lambda=0$ we set $0 \sq f = \rho_f$, where \[\rho_{f}(x) \coloneqq \lim_{\lambda \rightarrow \infty} \frac{f(\lambda x)}{\lambda}, \, x \neq 0,\] which is called the recession function of $f$. Geometrically, \[ \epi(\lambda \sq f)=\lambda \epi(f)\quad\text{for}~\lambda>0.\] 

The following properties of the Fenchel-Legendre transform are well-known (see, for example, \cite{RockafellarConvex1997}).
\begin{proposition}\label{proposition:properties_transform}
     Let $f, g \in \fconv, \lambda,\eta \geq 0$, $X=(x,t)\in \Rnn$ and $K \in \Kn$. Then 
\begin{itemize}[itemsep=0pt]
    \item[i)] $(f^*)^*=f$,
    \item[ii)] $(f\vee g)^*=f^* \wedge g^*$,
    \item[iii)] $(\lambda f + \eta g)^*=(\lambda \sq f^*)\square (\eta \sq g^*)$,
    \item[iv)] $f \geq g \Rightarrow f^* \leq g^*$,
    \item[v)]  $(f(\cdot -x)+t)^*=f^*-t+\langle x, \cdot \rangle$,
    \item[vi)] $(h_K)^*=I_K$.
\end{itemize}
\end{proposition}\noindent

Recall that the recession cone of a set $A \subset \R^{n+1}$ is the maximal cone $\rec(A)$ such that $A+ \rec(A) \subseteq A$. If we set $A=\epi(f^*)$, then $\rec(A)$ corresponds to the epigraph of the recession function $\rho_{f^*}$. 
\begin{theorem}[\cite{RockafellarConvex1997} Theorem 13.3]
\label{thm:rec}
    Let $f \in \fconv$. The support function of $\dom f$ is then the recession function $\rho_{f^*}$ of $f^*$, and the support function of $\dom f^*$ is the recession function of $\rho_{f}$ of $f$. 
\end{theorem}\noindent
In particular, domain restrictions on a function $f\in\fconv$ translate to lower and upper bounds on the recession cone of $\epi(f^*)$. This allows us to give the dual definition of an $(\A,\Oc)$-cone. Denote by $\mathfrak{A}$ and $\mathfrak{O}$ the convex cones in $\R^{n+1}$ given by
\begin{equation}\label{eq:cones}
    \mathfrak{A}=\epi(h_\A), \, \mathfrak{O}=\epi(h_\Oc).
\end{equation}
We say that a subset $\tilde{\cone}\subset \fconv$ is an $(\mathfrak{A}, \mathfrak{O})$-cone if
\begin{itemize}
	\item the set $\tilde{\cone}$ is a cone with respect to infimal convolution and epi-multiplication with $\lambda>0$,
	\item the inclusions $\mathfrak{A} \subset \rec(\epi(f)) \subset \mathfrak{O}$ hold for every $f\in\tilde{\cone}$,
	\item the functions $h_\A$ and $h_\Oc$ belong to $\tilde{\cone}$,
	\item for every $K\in\mathcal{K}^{n+1}$, the function $\lfloor K \rfloor  \square h_\A$ belongs to $\tilde{\cone}$. 
\end{itemize}
Notice that the last property is equivalent to the following: For every $K \in \KN$, the function $\lfloor K + \mathfrak{A} \rfloor$ belongs to the cone $\tilde{\cone}$ (compare \eqref{eq:floor_function} for the definition of $\lfloor \cdot \rfloor$, and note that this definition extends to more general closed convex sets, see \cite[Theorem 5.3]{RockafellarWetsVariationalAnalysis1998}). 
\begin{remark}
    The choice $\mathfrak{A}=\mathfrak{O}$ gives a family of functions with fixed recession cones. Up to a translation of the epi-graph, this corresponds to the notion of pseudo-cones introduced by Schneider \cite{SchneiderPseudocones2024}.
\end{remark}

The Fenchel-Legendre transform establishes a continuous bijection between the family of finite convex functions $\fconvf$, and the family of super-coercive convex functions
\begin{align*}
    \fconvs\coloneqq \left\{f\in\fconv:\lim_{|x|\rightarrow\infty}\frac{f(x)}{|x|}=\infty\right\},
\end{align*}
where both spaces are equipped with the topology induced by epi-convergence, compare, for example, \cite{ColesantiEtAlhomogeneousdecompositiontheorem2020}. Moreover, property $ii)$ in Proposition \ref{proposition:properties_transform} shows that a valuation on the former space induces a valuation on the latter and vice-versa. Due to  property $v)$ in Proposition \ref{proposition:properties_transform}, we introduce the following notion: For $f\in \fconv$ and $X=(x,t)\in \R^{n+1}$ let $\tau_Xf\in \fconv$ denote the function obtained from $f$ by an epi-translation by $X$, that is, 
\begin{align*}
    (\tau_Xf)(y)=f(y-x)+t.
\end{align*}
Let $\tilde{\cone}$ denote an $(\mathfrak{A}, \mathfrak{O})$-cone. We call a valuation $Z:\tilde{\cone}\rightarrow V$ into a Hausdorff topological vector space polynomial of degree at most $d$ (with respect to epi-translations) if the map
\begin{align*}
    X\mapsto Z(\tau_Xf)
\end{align*}
is a polynomial of degree at most $d$ in $X\in\R^{n+1}$ for all $f\in \tilde{\cone}$. The previous discussion and Proposition \ref{proposition:properties_transform} thus imply the following.
\begin{proposition}\label{proposition:dual_valuations}
     Let $\cone$ be an $(\A,\Oc)$-cone, and $\mathfrak{A}$ and $\mathfrak{O}$ the two closed and convex cones in $\Rnn$ satisfying \eqref{eq:cones}. Then the cone \[\tilde\cone \coloneqq \{f \in \fconv : f=g^*, g \in \cone \}\] is an $(\mathfrak{A},\mathfrak{O})$-cone. Moreover, a functional $Z:\cone\rightarrow V$ into a Hausdorff topological vector space belongs to $\PV_d(\Gamma,V)$, if and only if the functional
     \begin{align*}
         \tilde Z : \tilde \cone &\rightarrow V\\
         f &\mapsto Z(f^*)
     \end{align*}
     is a valuation on $\tilde \cone$ which is continuous and polynomial of degree at most $d$ with respect to epi-translations.
\end{proposition}
 We say that a functional $Z:\tilde{\cone}\rightarrow V$ is epi-homogeneous of degree $k \in \N$ if \[ Z(\lambda \sq f)=\lambda^i Z(f)\] for every $f \in \tilde{\cone}$ and $\lambda \geq 0$. Proposition \ref{proposition:properties_transform} shows that this corresponds to the previous notion of homogeneity for valuations on $\cone$ under the Fenchel-Legendre transform. 
The following is the dual version of Theorem \ref{maintheorem:HomogeneousDecomposition}.
\begin{theorem}
		\label{maintheorem:HomogeneousDecomposition_geom}
		Let $\tilde{\cone}$ be an $(\mathfrak{A},\mathfrak{O})$-cone, $V$ a Hausdorff topological vector space, and $Z:\tilde{\cone}\rightarrow V$ a continuous valuation that is polynomial of degree at most $d$ with respect to epi-translations. Then there exist unique continuous valuations $Z_k:\tilde{\cone}\rightarrow V, k \in \{0,\dots,n+d \}$, such that $Z_k$ is polynomial of degree at most $d$ with respect to epi-translations, epi-homogeneous of degree $k$, and
		\begin{align*}
			Z=\sum_{k=0}^{n+d}Z_k.
		\end{align*}
	\end{theorem}
  
    Let $Z:\tilde{\cone}\rightarrow V$. We say that a closed set $A\subset \R^n$ supports $Z$ if it has the following property: If $U$ is an open neighborhood of $A$ and $f,g\in \tilde{\cone}$ are two functions such that 
    \begin{align}\label{eq:supp_geom}
        \text{for every } y \in U \text{ there exists } x \in \R^n \text{ such that }y \in \partial f(x) \cap \partial g(x)\text{ and }f(x)=g(x),
    \end{align}
    then $Z(f)=Z(g)$. Notice that \eqref{eq:supp_geom} by Lemma \ref{lemma:PropertiesSubgradients} implies that for every $y \in U$ \[ f^*(y)=\langle x,y \rangle - f(x)= \langle x, y \rangle - g(x)=g^*(y).\] That is, $f^*$ and $g^*$ coincide on the open set $U$. On the other hand, if $f^*(y)=g^*(y)$ for every $y \in U$, there exists $x \in \R^n$ such that $x \in \partial f^*(y) \cap \partial g^*(y)$ with the desired property by Lemma \ref{lemma:PropertiesSubgradients}. Thus, the definition of support given above corresponds precisely to the description of the support given in Proposition \ref{proposition:CharacterizationSupport}. We thus obtain the following dual version of Theorem \ref{maintheorem:extension}.
 \begin{theorem}
		\label{maintheorem:extension_geom}
        Let $(\A,\Oc)$ and $(\mathfrak{A},\mathfrak{O})$ be defined as before, and let $V$ be a locally convex vector space. Consider a continuous valuation $Z: \fconvcd \rightarrow V$ that is polynomial of degree at most $d$ with respect to epi-translation, and such that $\supp Z$ is compact and satisfies $\supp Z\subset \mathrm{int}\Oc$. For every $(\mathfrak{A},\mathfrak{O})$-cone $\tilde{\cone}$ there exists a unique valuation $Z_{\tilde{\cone}}: \tilde\cone \rightarrow V$, continuous and polynomial of degree at most $d$ with respect to epi-translation, such that
		\begin{align*}
			Z(f)=Z_{\tilde{\cone}}(f \square h_\A)\quad\text{for every}~f\in\fconvcd.
		\end{align*}
	\end{theorem}
    Finally, notice that both $\fconvcd$ and $\fconvs$ are $(\mathfrak{A},\mathfrak{O})$-cones for \[\mathfrak{A}=\mathfrak{O}=\{(0,t) \in \Rn \times \R \colon t \geq 0 \}.\] We have the following dual version of Corollary \ref{maincorollary:extensionRealValuations}.
	\begin{corollary}
        Let $Z:\fconvcd \rightarrow \R$ be a continuous polynomial valuation of degree at most $d$. Then $Z$ extends uniquely by continuity to a continuous valuation on $\fconvs$.
	\end{corollary}

\section*{Acknowledgements} The second author was supported by the Austrian Science Fund (FWF): 10.55776/P34446 and, in part, by the Gruppo Nazionale per l’Analisi Matematica, la Probabilit\'a e le loro Applicazioni (GNAMPA) of the Istituto Nazionale di Alta Matematica (INdAM). 
 
\footnotesize
\bibliography{references}
\bibliographystyle{siam}

\parbox[t]{8.5cm}{
Jonas Knoerr\\
Institut f\"ur Diskrete Mathematik und Geometrie\\
TU Wien\\
Wiedner Hauptstra{\ss}e 8-10, E104-6\\
1040 Wien, Austria\\
e-mail: jonas.knoerr@tuwien.ac.at}

\bigskip

\parbox[t]{8.5cm}{
Jacopo Ulivelli\\
Institut f\"ur Diskrete Mathematik und Geometrie\\
TU Wien\\
Wiedner Hauptstra{\ss}e 8-10, E104-6\\
1040 Wien, Austria\\
e-mail: jacopo.ulivelli@tuwien.ac.at}

\end{document}